\documentclass{amsart}

\usepackage{graphicx, amsmath, tikz-cd, array, amssymb,  enumitem, multirow, float, comment, xcolor, mathrsfs}

\usepackage[section]{placeins}

\usepackage{subcaption}

\usepackage{mathabx}

\usepackage[square, numbers]{natbib}

\usepackage{dsfont} 

\usepackage[colorlinks = true, citecolor = blue, urlcolor = blue]{hyperref}

\newtheorem{theorem}{Theorem}[section]
\newtheorem{lemma}[theorem]{Lemma}
\newtheorem{corollary}[theorem]{Corollary}
\newtheorem{prop}[theorem]{Proposition}

\theoremstyle{definition}
\newtheorem{definition}[theorem]{Definition}
\newtheorem{example}[theorem]{Example}

\theoremstyle{remark}
\newtheorem{remark}[theorem]{Remark}

\newcommand{\PH}[1]{\mathrm{{PH}}}

\def\C{\mathbb{C}}
\def\F{\mathbb{F}}

\def\N{\mathbb{N}}
\def\Q{\mathbb{Q}}
\def\R{\mathbb{R}}

\def\Z{\mathbb{Z}}

\def\xx{\mathbf{x}}
\def\yy{\mathbf{y}}

\def\II{\mathbf{I}}
\def\CC{\mathbf{C}}
\def\JJ{\mathbf{J}}
\def\NN{\mathbf{N}}
\def\RR{\mathbf{R}}

\def\bcd{\mathsf{bcd}}

\def\<{\langle}
\def\>{\rangle}

\DeclareMathOperator*{\argmax}{\arg\!\max\;}

\begin{document}

\title{K{\"u}nneth Formulae in Persistent Homology}

\author[Hitesh Gakhar]{Hitesh Gakhar}
\address{
\shortstack[l]{
Department of Mathematics, \\
Michigan State University \\
East Lansing, MI, USA.}}
\email{gakharhi@msu.edu}

\author[Jose Perea]{Jose A. Perea }
\address{
\shortstack[l]{
Department of Computational Mathematics, Science \& Engineering \\
Department of Mathematics, \\
Michigan State University \\
East Lansing, MI, USA.}}
\email{joperea@msu.edu}

\thanks{This work was partially supported by the NSF (DMS-1622301)}




\date{}

\dedicatory{}

\begin{abstract} The classical K\"{u}nneth formula in algebraic topology describes the homology of a product space in terms of that of its factors. In this paper, we prove K\"{u}nneth-type theorems
for the persistent homology of the categorical and tensor  product of filtered spaces.
That is,  we describe the persistent homology of these  product filtrations in terms of that of the filtered components.
In addition to comparing the two products, we present two applications in the setting of Vietoris-Rips complexes: one towards more efficient  algorithms for product metric spaces with the maximum metric, and the other recovering persistent homology  calculations on the $n$-torus.
\end{abstract}

\maketitle
\tableofcontents
\section{Introduction}

Persistent homology is an algebraic and computational tool from topological data analysis \cite{perea2018brief}.
Broadly speaking, it is used to quantify multiscale features of shapes and
some  of its applications to science and engineering include
coverage problems in sensor networks  \cite{sensor},
recurrence detection in time series data  \cite{toroidal, perea2019notices, sw1pers},
the identification  of spaces of fundamental features from natural scenes  \cite{naturalimages},
inferring spatial properties of unknown environments via biobots  \cite{biobot},
and more.
Some of the underlying ideas of persistence are also starting to permeate pure mathematics,
most notably  symplectic geometry  \cite{hofer, floernovikov}.

The successes  of persistent homology stem in part  from a strong theoretical foundation \cite{CPH, dualities, oudotBook}
and a focus on efficient algorithmic implementations \cite{phchunks, distributed, ph1pass, phtwist, effcubic}.
That said, the inherent algorithmic problems are far from being solved---see \cite{roadmap} for a recent survey---and
current theoretical approaches to abstract computations are limited to very specific cases \cite{adams, vrellipse, vrpolygon}.
Our goal in this paper is to add to the toolbox of techniques for abstract computations of persistent homology
when the input  can be described in terms of (a product of) simpler components.

To be more specific,  the simplest input to a persistent homology
computation is a collection $\mathcal{X} = \{ X_0 \subset X_1 \subset \cdots \}$ of  topological  spaces
so that each inclusion $X_i \hookrightarrow X_{i+1}$ is continuous.
This is called a filtration.
Taking singular homology in dimension $n \geq 0$ with coefficients in a field $\F$ yields a diagram
\begin{equation}\label{eq:persistentHomology}
  H_n(X_0 ; \F) \xrightarrow{\;\;L_0\;\;} H_n(X_1 ; \F) \xrightarrow{\;\;L_1\;\;} \cdots
  \xrightarrow{\;L_{i-1}\;} H_n(X_i;\F) \xrightarrow{\;\;L_i\;\;} \cdots
\end{equation}
of vector spaces and linear maps between them. Here $L_i$ is the linear transformation induced by the inclusion $X_i \hookrightarrow X_{i+1}$.
A theorem of Crawley-Boevey \cite{crawley2015decomposition} implies that
if each $H_n(X_i ; \F)$ is finite dimensional---this is called being \emph{pointwise finite}---then there exists
a set
of possibly repeating intervals $[\ell, \rho) \subset \N := \{0,1,2, \ldots \}$
  (i.e., a \emph{multiset}), which  uniquely determines  the isomorphism type
of (\ref{eq:persistentHomology}).
The resulting multiset---this is the output of the persistent homology computation---is denoted $\bcd_n(\mathcal{X}; \F)$,
or just $\bcd_n(\mathcal{X})$ if there is no ambiguity,
and it is called the barcode of (\ref{eq:persistentHomology}).
An interval $[\ell, \rho ) \in \bcd_n(\mathcal{X})$ corresponds to a class $\eta \in H_n(X_\ell ; \F)$ which is not
in the image of $L_{\ell-1}$, and  so that  $\rho$ is either $\infty$, or the smallest integer greater than  $ \ell$
for which
$\eta \in \mathsf{ker}(L_{\rho-1} \circ \cdots \circ L_{\ell})$.
Either way, $\rho- \ell$ is the persistence of the homological feature $\eta$,
the index $\ell$ is its birth-time, and $\rho$ is its death-time.
Our goal  is to understand  diagrams like (\ref{eq:persistentHomology}),
and their barcodes, for the case when the input is the product of two filtrations.

What do we mean by  product filtrations?
We consider two answers, one categorical/computational and the other algebraic.
The categorical/computational answer is to let
$\mathcal{X}\times \mathcal{Y}$ be the filtration  $X_0 \times Y_0 \,\subset\, X_1 \times Y_1 \subset \cdots $
 or said more succinctly,
\begin{equation}\label{eq:categoricalProduct}
\left(\mathcal{X} \times \mathcal{Y}\right)_k\, := \, X_k \times Y_k \;\;\;\;\;\; , \;\;\;\;\;\; k\in \N.
\end{equation}
This answer is categorical in the sense that if
$X_k$ and $Y_k$ are objects in a category $\CC$ with all finite products, and the indices $k$ are  objects in a small category $\II$,
then (\ref{eq:categoricalProduct}) is the product (i.e., it satisfies the appropriate universal property) in the  category of functors from
$\II$ to $\CC$ (see Proposition \ref{catprod}).
Thus, we call $\mathcal{X}\times \mathcal{Y}$ the \emph{categorical product} of the filtrations $\mathcal{X}$ and $\mathcal{Y}$.
This answer is also computational in the sense that it is relevant to persistent homology algorithms:
if $(X,d_X)$ is a metric space and $R_\epsilon(X, d_X)$ is its $\epsilon$-Rips complex---that is,
the abstract simplicial complex whose simplices  are the finite nonempty subsets of $X$ with diameter less than $\epsilon$---then
(see Lemma \ref{lemma: rips})
\begin{equation}\label{eq:ProductRips}
R_\epsilon(X\times Y, d_{X\times Y}) = R_\epsilon(X,d_X)\times R_\epsilon(Y,d_Y)
\end{equation}
where $d_{X\times Y}$ is the maximum metric
\begin{equation}\label{eq:MaxMetric}
d_{X\times Y}\big((x,y), (x',y')\big) := \max\{d_X(x, x'), d_Y(y,y')\}
\end{equation}
and the product of  Rips complexes on the right hand side of (\ref{eq:ProductRips}) takes place in the category of abstract simplicial complexes
(see \ref{expl:Simp} and \ref{def:Simp} for definitions).

The second answer to seeking an appropriate notion of product filtration is algebraic in a sense which will be clear below (equation (\ref{eq:IntroSequence})).
 For now, let us just define what it is:  the \emph{tensor product} of   $\mathcal{X}$ and $\mathcal{Y}$ is the filtration
\begin{equation}\label{eq:tensorProduct}
\left(\mathcal{X} \otimes \mathcal{Y}\right)_k \, := \, \bigcup_{i+j=k} \left( X_i \times Y_j \right) \;\;\;\;\; , \;\;\;\; k\in \N.
\end{equation}

The simplest versions of the persistent K\"unneth formulae we prove in this paper (Theorems \ref{barcodeDPF} and \ref{KunnethTPF})
can  be stated as follows:
\begin{theorem}\label{thm:IntroKunneth}
Let $\mathcal{X}$ and $\mathcal{Y}$ be pointwise finite filtrations.
Then, the barcodes of the categorical product $\mathcal{X}\times\mathcal{Y}$
are given by the (disjoint)  union of multisets
\begin{align*}
\bcd_n\left(\mathcal{X} \times \mathcal{Y} \right)
=
\bigcup_{i + j = n}
\Big\lbrace
I \cap J \ \Big\vert \ I \in \bcd_{i}(\mathcal{X})\,, \; J \in \bcd_{j}(\mathcal{Y})
\Big\rbrace.
\end{align*}
Similarly,  the barcodes for the tensor product filtration $\mathcal{X} \otimes \mathcal{Y}$  satisfy
\begin{align*}
\bcd_n(\mathcal{X} \otimes \mathcal{Y})
=
\bigcup_{i + j = n} 	\bigg\{  (\ell_J + I) \cap (\ell_I + J)\ \Big\vert \ I \in \bcd_i(\mathcal{X})\,,\; J \in \bcd_j(\mathcal{Y}) \bigg\} \hspace{1cm}\\
\bigcup \hspace{5.15cm}\\
\bigcup_{i + j = n-1}\bigg\{ (\rho_J + I)\cap(\rho_I + J)\Big\vert \ I \in {\mathsf{bcd}}_i(\mathcal{X})\,,\; J \in {\mathsf{bcd}}_j(\mathcal{Y}) \bigg\}\hspace{1cm}
\end{align*}
where $\ell_J$ and $\rho_J$ denote, respectively, the left and right endpoints of the interval $J$.
\end{theorem}

It is worth noting that this theorem   also holds true for filtrations of (ordered) simplicial complexes, simplicial sets,
simplicial homology,   and  products in the appropriate categories.
In particular,  it implies the following corollary (\ref{cor:RipsKunneth}) for the Rips persistent homology of product metric spaces:
\begin{corollary}
Let $(X,d_X), (Y,d_Y)$ be finite metric spaces
and let $\bcd_n^\mathcal{R}(X,d_X)$ be the barcode of the Rips filtration
$\mathcal{R}(X, d_X) := \{R_\epsilon(X,d_X)\}_{\epsilon \geq 0}$. Then,
\[
\bcd_n^\mathcal{R}(X\times Y,d_{X\times Y}) = \bigcup_{i + j = n} \Big\{ I \cap J \; \Big| \; I \in \bcd^\mathcal{R}_i(X,d_X) \, , \; J \in \bcd_j^\mathcal{R}(Y,d_Y)\Big\}
\]
for all $n\in\N$, if $d_{X\times Y}$ is the maximum metric (\ref{eq:MaxMetric}).
\end{corollary}
\noindent We envision for this type of result to be used in abstract persistent computations, as well as
in the design of new and more efficient persistent homology algorithms.

Let us  see an example  illustrating Theorem \ref{thm:IntroKunneth} and contrasting the categorical and tensor product filtrations.
Indeed, consider the filtered simplicial complexes $\mathcal{ K }$ and $\mathcal{ L }$ shown in Figure \ref{fig:filteredK}.
\begin{figure}[!htb]
\centering
	\includegraphics[width=0.7\textwidth]{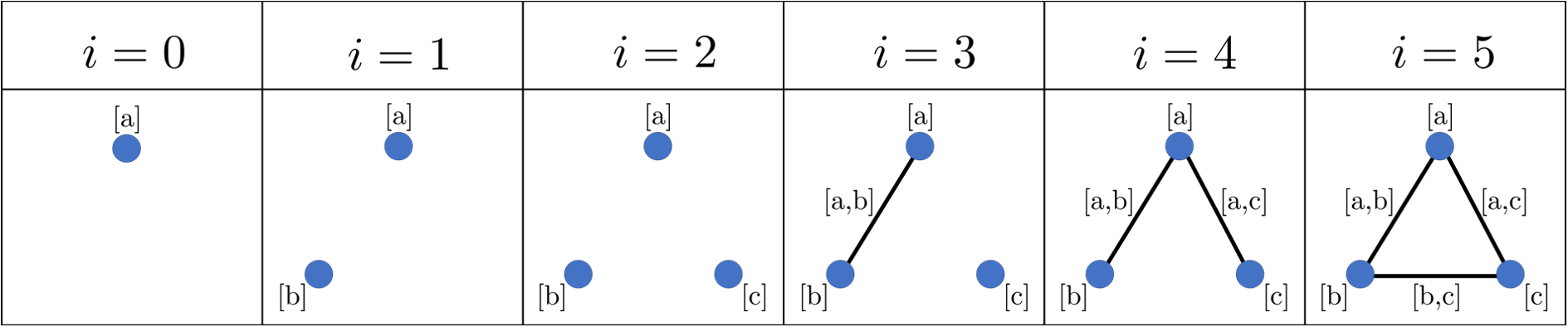}
	\caption{Filtered simplicial complex $\mathcal{K}=\mathcal{L}$, with $ K_i = K_5$ for $i\geq 5$.}
	\label{fig:filteredK}
\end{figure}

The  two  filtrations $\mathcal{K} \times \mathcal{L}$ and $\mathcal{K} \otimes \mathcal{L}$ are shown in Figure \ref{fig:filteredKL}.
The products  $K_i\times L_j$  are computed in the category of \emph{ordered  simplicial complexes}
using the order $a \leq b \leq c$
(see \ref{expl:oSimp} and \ref{def:oSimp} for definitions).

\begin{figure}[!htb]
\centering
\begin{subfigure}{0.4\textwidth}
    \centering
	\includegraphics[width= \textwidth]{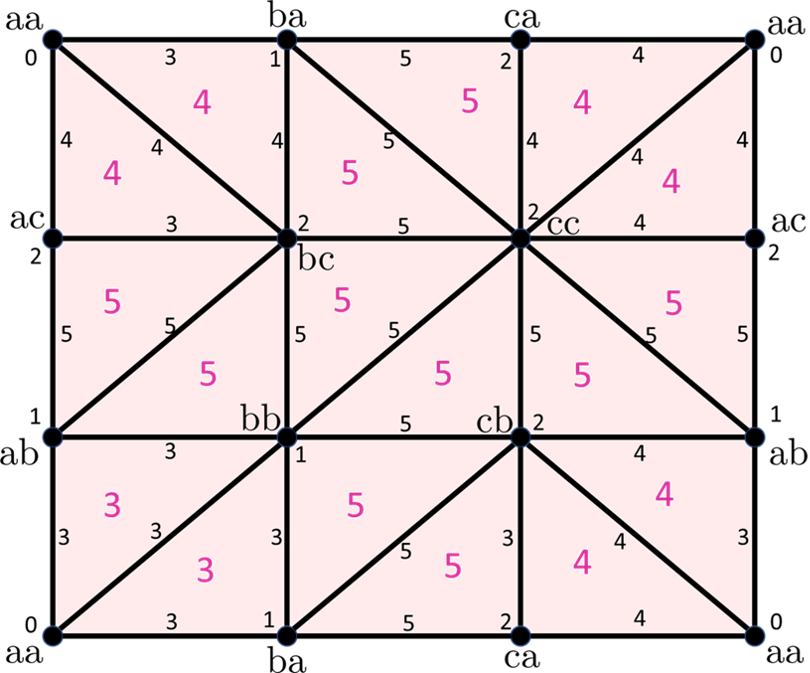}
\end{subfigure}
\qquad
\begin{subfigure}{0.4\textwidth}
    \centering
	\includegraphics[width= \textwidth]{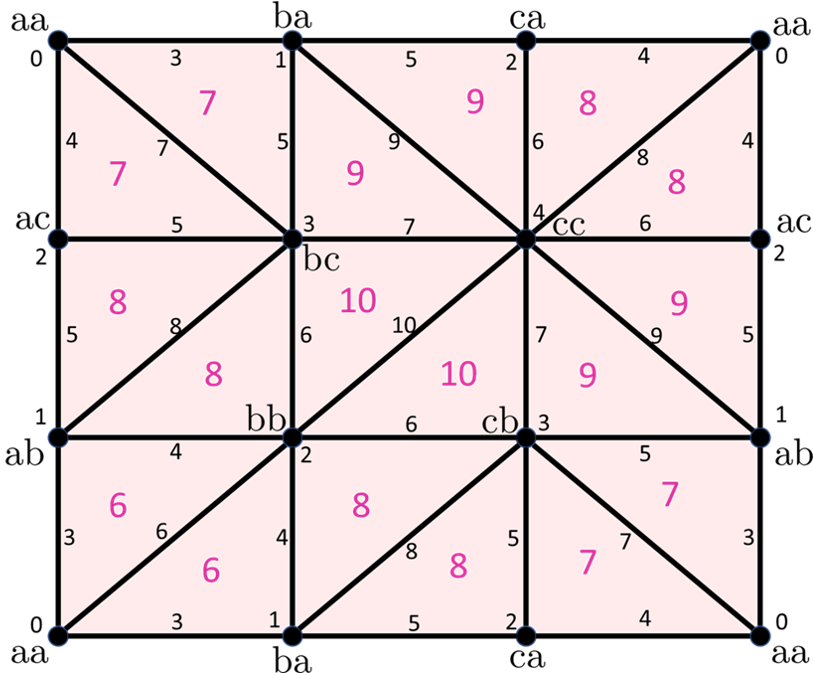}
\end{subfigure}
	\caption{The filtrations    $\mathcal{K} \times \mathcal{L}$ (left) and $\mathcal{K} \otimes \mathcal{L}$ (right). The numbers denote the filtration index at which the corresponding simplices enter the filtration.}
	\label{fig:filteredKL}
\end{figure}

When comparing $\mathcal{K}\times \mathcal{L}$ and $\mathcal{K}\otimes \mathcal{L}$, the first thing to note is the heterogeneity
of indices in   $\mathcal{K}\otimes \mathcal{L}$.
This suggests that the tensor product filtration has  more short-lived birth-death events than the categorical product,
which is supported by the formulas in Theorem \ref{thm:IntroKunneth}.
That is, the barcodes of $\mathcal{K}\otimes \mathcal{L}$ are expected to be ``noisier'' (i.e., with more short intervals) than those of $\mathcal{K}\times \mathcal{L}$.
The barcodes for $\mathcal{K}\times \mathcal{L}$ and $\mathcal{K}\otimes \mathcal{L}$ from Figure \ref{fig:filteredKL}, over any field  $\F$,
   are shown in Figure \ref{fig:barcodes}.
\begin{figure}[!htb]
	\includegraphics[width=\linewidth]{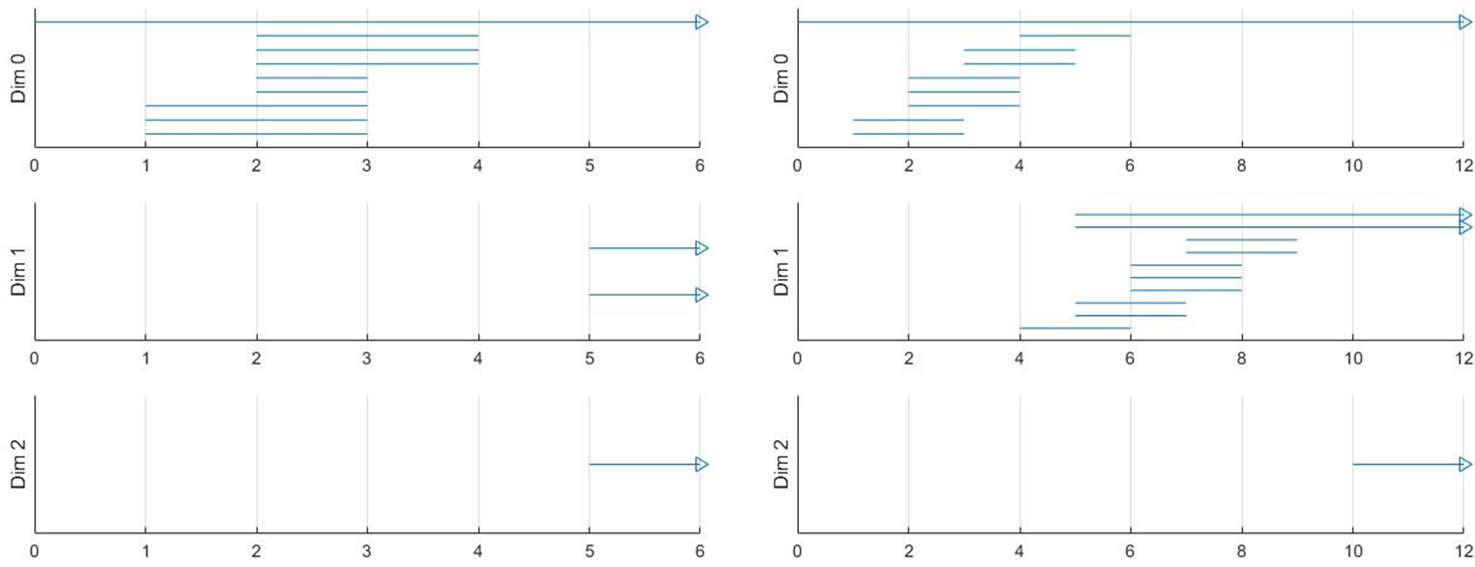}
	\caption{Barcodes for $\mathcal{K} \times \mathcal{L}$ (left) and $\mathcal{K}\otimes\mathcal{L}$ (right) in dimensions $0$, $1$, and $2$.}
	\label{fig:barcodes}
\end{figure}

As described in Theorem \ref{thm:IntroKunneth},
these barcodes are  related to those of  the complexes   $\mathcal{K}$ and $\mathcal{L}$,
$\bcd(\mathcal{K}) =  \lbrace [0,\infty)_0, [1,3)_0, [2,4)_0, [5,\infty)_1  \rbrace  = \bcd(\mathcal{L})$, where the subscripts denote
homological dimension, as shown in
Tables \ref{table: diagonalbarcodes} and \ref{table: tensorbarcode}.
\begin{table}[!htb]
\parbox{.43\linewidth}{
\centering
		\newcolumntype{L}{>{\centering\arraybackslash}m{2.9cm}}
		\newcolumntype{l}{>{\centering\arraybackslash}m{1.7cm}}
		\begin{tabular}{| l | L | @{}m{0pt}@{}}
			\hline
			$\bcd(\mathcal{K}\times \mathcal{L})$ &  $\bcd(\mathcal{K}) \;\; \big| \;\; \bcd(\mathcal{L})$ & \\[.2cm]
			\hline
			$[5,\infty)_2$ & $[5,\infty)_{1}\cap [5,\infty)_{1}$&\\[.2cm]
			\hline
			$[5,\infty)_1$ & $[5,\infty)_{1}\cap[0,\infty)_{0}$&\\[.2cm]
			\hline
			$[5,\infty)_1$ & $[0,\infty)_{0}\cap[5,\infty)_{1}$&\\[.2cm]
			\hline
			$[0,\infty)_0$ & $[0,\infty)_{0}\cap[0,\infty)_{0}$&\\[.2cm]
			\hline
			$[2,4)_0$ &  $[2,4)_{0}\cap[2,4)_{0}$&\\[.2cm]
			\hline
			$[2,4)_0$ &  $[2,4)_{0}\cap[0, \infty)_{0}$&\\[.2cm]
			\hline
			$[2,4)_0$ & $[0,\infty)_{0}\cap[2,4)_{0}$&\\[.2cm]
			\hline
			$[1,3)_0$ &  $[1,3)_{0},\cap[1,3)_{0}$&\\[.2cm]
			\hline
			$[1,3)_0$ &  $[1,3)_{0}\cap[0, \infty)_{0}$&\\[.2cm]
			\hline
			$[1,3)_0$ &  $[0,\infty)_{0}\cap[1,3)_{0}$&\\[.2cm]
			\hline
			$[2,3)_0$ &  $[2,4)_{0}\cap[1,3)_{0}$&\\[.2cm]
			\hline
			$[2,3)_0$ &  $[1,3)_{0}\cap[2,4)_{0}$&\\[.2cm]
			\hline
		\end{tabular}
\caption{}
\label{table: diagonalbarcodes}
}
\hfill
\parbox{.55\linewidth}{
\centering
	\newcolumntype{L}{>{\centering\arraybackslash}m{4.2cm}}
	\newcolumntype{l}{>{\centering\arraybackslash}m{1.8cm}}
	\begin{tabular}{| l | L | @{}m{0pt}@{}}
		\hline		$\bcd_1(\mathcal{K}\otimes \mathcal{L})$ &  $\bcd(\mathcal{K}) \;\; \big| \;\; \bcd(\mathcal{L})$ &\\[.4cm]
		\hline
		$[5,\infty)_1$ &  $\left(0+[5,\infty)_{1}\right)\cap\left(5+[0,\infty)_{0}\right)$ &\\[0.2cm]
		\hline
		$[5,\infty)_1$ & $\left(5+[0,\infty)_{0}\right)\cap\left(0+[5,\infty)_{1}\right)$&\\[.2cm]
		\hline
		$[7,9)_1$  &  $\left(2+[5,\infty)_{1}\right)\cap\left(5+[2,4)_{0}\right)$&\\[.2cm]
		\hline
		$[7,9)_1$ &  $\left(5+[2,4)_{0}\right)\cap\left(2+[5,\infty)_{1}\right)$&\\[.2cm]
		\hline
		$[6,8)_1$ &  $\left(1+[5,\infty)_{1}\right)\cap\left(5+[1,3)_{0}\right)$ &\\[.2cm]
		\hline
		$[6,8)_1$ &  $\left(5+[1,3)_{0}\right)\cap\left(1+[5,\infty)_{1}\right)$&\\[.2cm]
		\hline
		$[6,8)_1$ & $\left(4+[2,4)_{0}\right)\cap\left(4+[2,4)_{0}\right)$&\\[.2cm]
		\hline
		$[5,7)_1$ &  $\left(3+[2,4)_{0}\right)\cap\left(4+[1,3)_{0}\right)$&\\[.2cm]
		\hline
		$[5,7)_1$ &  $\left(4+[1,3)_{0}\right)\cap\left(3+[2,4)_{0}\right)$&\\[.2cm]
		\hline
		$[4,6)_1$ &  $\left(3+[1,3)_{0}\right)\cap\left(3+[1,3)_{0}\right)$&\\[.2cm]
		\hline
	\end{tabular}
	\caption{}
	\label{table: tensorbarcode}
}
\end{table}

\subsection{Sketch of proof}
In establishing the formulae, the categorical product is the simpler of the two.
Given two  diagrams  of topological spaces and continuous maps
\begin{eqnarray*}
\mathcal{X} &=& \{  f_{\alpha' ,\alpha}: X_\alpha \rightarrow X_{\alpha'}\}_{\alpha\preceq \alpha' \in \mathcal{P}} \\
\mathcal{Y} &=& \{  g_{\alpha' ,\alpha}: Y_\alpha \rightarrow Y_{\alpha'}\}_{\alpha\preceq \alpha' \in \mathcal{P}}
\end{eqnarray*}
indexed by a separable toally ordered set  $(\mathcal{P} ,\preceq)$ --- e.g., the reals ---
 so that $f_{\alpha,\alpha}$ is the identity of $X_\alpha$ and
$f_{\alpha'',\alpha'}\circ f_{\alpha',\alpha} = f_{\alpha'',\alpha}$ for
$\alpha \preceq \alpha' \preceq \alpha''$ (similarly for $g_{\alpha',\alpha}$),
we let $\mathcal{X}\times \mathcal{Y} = \{f_{\alpha',\alpha} \times g_{\alpha',\alpha} : X_\alpha \times Y_\alpha \longrightarrow X_{\alpha'}\times Y_{\alpha'}\}_{\alpha \preceq \alpha' \in \mathcal{P}}$.
The classical topological K\"unneth theorem implies that the
 induced $\mathcal{P}$-indexed diagram
\[
\big\{
H_n(X_\alpha \times Y_\alpha; \F) \longrightarrow H_n(X_{\alpha'}\times Y_{\alpha'} ;\F)
\big\}_{\alpha \preceq \alpha' \in \mathcal{P}}
\]
is isomorphic to
\[
\left\{
\bigoplus_{i+j =n}H_i(X_\alpha;\F) \otimes_\F H_j( Y_\alpha; \F)
\longrightarrow
\bigoplus_{i+j = n }H_i(X_{\alpha'};\F)\otimes_\F H_j(Y_{\alpha'} ;\F)
\right\}_{\alpha \preceq  \alpha' \in \mathcal{P}}
\]
and  a barcode computation, in the pointwise finite case, yields the result.

The tensor product of   $\mathcal{X} = \{X_i \subset X_{i+1}\}_{i\in \N}$
and  $\mathcal{Y} = \{Y_i \subset  Y_{i+1}\}_{i\in \N}$ requires a bit more work.
Define the persistent homology of $\mathcal{X}$  as
\[
PH_n(\mathcal{X} ; \F ) := \bigoplus\limits_{i\in \N} H_n(X_i ; \F)
\]
This object has the structure of  a graded module over the polynomial ring
$\F[t]$ in one variable $t$,
where the product of $t$ with a homogeneous element of degree $i$ reduces to
applying the linear transformation induced by the inclusion $X_i \hookrightarrow X_{i+1}$.
For purely algebraic reasons---specifically the algebraic K\"unneth theorem for chain complexes of flat modules over a PID---and
a new persistent version of the Eilenberg-Zilber theorem    (Theorem \ref{EZnew}),
one obtains
 a natural short exact sequence
\begin{align}\label{eq:IntroSequence}
\begin{split}
0 \rightarrow
\bigoplus_{i+j = n} PH_i(\mathcal{X}; \F) \otimes_{\F[t]} PH_j&(\mathcal{Y};\F)
\rightarrow
PH_n(\mathcal{X}\otimes \mathcal{Y};\F)
\rightarrow \\
& \bigoplus_{i+j = n}
\mathsf{Tor}_{\F[t]} (PH_i(\mathcal{X};\F), PH_{j-1}(\mathcal{Y};\F))
\rightarrow 0
\end{split}
\end{align}
which splits, though not naturally.
This is why we think of $\mathcal{X}\otimes \mathcal{Y}$ as the algebraic answer
to asking what an appropriate  notion  of  product filtration is:
unlike $\mathcal{X\times  Y}$, the tensor product $\mathcal{X \otimes Y}$ fits into the type of short exact sequence
one would expect  in a K\"unneth theorem for persistent homology.
The barcode formula in Theorem \ref{thm:IntroKunneth}
follows, in the pointwise finite case, from the existence of  graded $\F[t]$-isomorphisms
\[
PH_i(\mathcal{X};\F)
\;\;\cong \;\;
\bigoplus\limits_{[\ell, \rho)\in \bcd_i(\mathcal{X};\F)} \left( t^\ell \cdot \F[t]\right)\big/\left(t^\rho\right)
\]
with the convention that $t^\infty = 0$, using that (\ref{eq:IntroSequence}) splits, and computing the tensor
and Tor $\F[t]$-modules explicitly  in terms of   operations on intervals.

We further establish that (\ref{eq:IntroSequence}) still  holds when the inclusions $X_i \hookrightarrow X_{i+1}$ and
$Y_i \hookrightarrow Y_{i+1}$ are replaced by continuous maps $f_i$ and $g_i$, respectively.
In this case $\otimes$ is replaced by a generalized tensor product $\otimes_{\mathbf{g}}$,   equivalent  to $\otimes$ for
inclusions, and   defined as follows:  the space
$(\mathcal{X}\otimes_{\mathbf{g}} \mathcal{Y})_k$ is the \emph{homotopy colimit} of the  functor
from the poset  $\triangle_k = \{(i,j) \in \N^2 : i + j \leq k\}$ to $\mathbf{Top}$ sending $(i,j)$ to $X_i\times Y_j$,
and $(i \leq i', j \leq j' )$ to the map $ (f_{i'-1}\circ \dots \circ f_i)\times (g_{j'-1}\circ \dots \circ g_j): X_i \times Y_j \rightarrow X_{i'}\times Y_{j'}$.
The map $(\mathcal{X}\otimes_{\mathbf{g}}\mathcal{Y})_k \rightarrow (\mathcal{X}\otimes_{\mathbf{g}}\mathcal{Y})_{k+1}$
is the one induced at the level of homotopy colimits by the inclusion $\triangle_k \subset \triangle_{k+1}$
of indexing posets.

\subsection{Organization of the paper}
Section \ref{background} is devoted to the algebraic background needed for the paper;
it takes  a categorical viewpoint and introduces notions such as persistent homology,
the classical K\"unneth theorems,
 and homotopy colimits.
In section \ref{filtrations} we define products of (diagrams of) spaces and study their properties,
while sections \ref{Dthm} and  \ref{TPthm}
contain  the proofs of our persistent K\"unneth formulae for the categorical and generalized tensor products, respectively.
In section \ref{application} we present two applications of the K\"{u}nneth formula for the categorical product in the setting of Vietoris-Rips complexes. The first application is to faster computations of the persistent homology of $\mathcal{R}(X\times Y,d_{X\times Y})$,
and the second application revisits theoretical results about the Rips persistent homology of the $n$-torus.

\subsection{Related work}
Different versions of the algebraic K\"{u}nneth theorem in persistence have appeared in the last two years. 
In \cite[Proposition 2.9]{algkunneth}, the authors prove a K\"{u}nneth formula for the tensor product of filtered chain complexes, 
while \cite[Section 10]{pershomologicalalgebra} establishes  K\"{u}nneth theorems for the graded tensor product and the sheaf tensor product of persistence modules. 
In \cite{summetric}, the authors prove a K\"{u}nneth formula relating the persistent homology of metric spaces $(X,d_X), (Y,d_Y)$ to the persistence of their cartesian product
 equipped with the $L^1$ (sum) metric $(X\times Y, d_X + d_Y)$, for homological dimensions $n=0,1$.
The persistent K\"unneth theorems proven here are the first to start at the level of filtered spaces, 
identifying the appropriate product filtrations and resulting barcode formulas. 
Compared to the   sum metric  \cite{summetric},  
we remark that our results for the maximum metric hold in all homological dimensions.

\section{Preliminaries}\label{background}
This section deals with   the necessary algebraic
background for later portions of the paper.
We hope   this will make the presentation more  accessible to a broader audience,
though experts should feel free to skip to Section \ref{filtrations} and come back  as necessary.
We provide a brief review of persistent homology, the rank invariant, the algebraic K\"unneth formula for the tensor product of  chain complexes of modules over a Principal Ideal Domain (PID), as well as an explicit
model for the homotopy colimit of a diagram of topological spaces.
For a more detailed treatment we refer
the interested reader to the references therein.

\subsection{Persistent homology}
The framework we use here is that  of diagrams
in a category (see also \cite{categorification}).
Besides the definitions and the classification via barcodes, the main point from this section is that the persistent homology of a diagram of spaces
is isomorphic to the standard homology of the associated persistence chain complex    (see \cite{CPH} and Theorem \ref{PH}).

\begin{definition}
If $\mathbf{C}$  and $\mathbf{I}$  are categories, with $\mathbf{I}$ small (i.e., so that its objects form a set), then we denote by $\mathbf{C}^{\mathbf{I}}$ the category whose objects are functors from $\mathbf{I}$ to $\mathbf{C}$, and whose morphisms are natural transformations between said functors.
The objects of the functor category
$\mathbf{C}^\mathbf{I}$ are often referred to in the literature as
$\mathbf{I}$-\emph{indexed diagrams in} $\mathbf{C}$.
Two diagrams $D,D' \in \mathbf{C^I}$ are said to be isomorphic, denoted $D \cong D'$, if they are naturally isomorphic as functors.
\end{definition}

Here is an example describing the main type of indexing category we will consider throughout the paper.

\begin{example}
	Let $(\mathcal{P}, \preceq)$ be a  partially ordered set (i.e., a poset). Let $\mathbf{P}$ denote the category whose objects are the elements  of $\mathcal{P}$, and a unique morphism $x \rightarrow y$ for each pair   $x \preceq y$ in $\mathcal{P}$.
We call $\mathbf{P}$ the \emph{poset category} of $\mathcal{P}$.
\end{example}

As for the target category $\mathbf{C}$, we will  mostly be interested in spaces and their algebraic invariants.

\begin{example}\label{Ftop}
If $\textbf{Top} $ is the category of topological spaces and continuous maps,
and $\II$ is a thin category (i.e., with at most one morphism between any two objects),
then each object of $\mathbf{Top^I}$ is a collection $\mathcal{X} = \lbrace  f_{j,i}: X_i \rightarrow X_j \rbrace_{i\rightarrow j \in \mathbf{I}}$ of topological spaces $X_i$, and  continuous maps  $f_{j,i}: X_i \longrightarrow X_j$ for each morphism $i \rightarrow j$ in $\II$ satisfying:
$f_{i,i}$ is the identity of $X_i$, and
$f_{k,j} \circ f_{j,i} = f_{k,i}$ for any $i \rightarrow j \rightarrow k$.
Similarly, a morphism $\phi$ from $\mathcal{ X } = \lbrace   f_{j,i}: X_i \rightarrow X_{j} \rbrace$ to $\mathcal{ Y } = \lbrace   g_{j,i}: Y_i \rightarrow Y_{j} \rbrace$ is a  family of maps $\phi_i: X_i \rightarrow Y_i$ such that  $ g_{j,i} \circ \phi_i = \phi_{j} \circ f_{j,i} $ for every $i\rightarrow j$.
\end{example}

\begin{example}\label{expl:MotivationTDA}
Let $\mathbb{N}= \{0,1,2,\ldots \}$   with its usual (total) order, and let $\NN$ be its poset category.
$\NN$-indexed diagrams in $\mathbf{Top}$ arise in Topological Data Analysis (TDA) as follows.
Let $\mathbb{M}$ be a metric space, let $\mathbb{X} \subset \mathbb{M}$ be a subspace and let $X\subset \mathbb{M}$ be  finite.
In applications $X$ is the data  one observes (e.g., images, text documents, molecular compounds, etc), obtained by sampling from/around  $\mathbb{X}$ (the  ground truth, which is unknown in practice),
both sitting in an ambient space $\mathbb{M}$.
With the goal of estimating the topology of $\mathbb{X}$ from $X$,
 one starts by letting $X^{(\epsilon)}$ be the union of open balls
in $\mathbb{M}$ of radius $\epsilon\geq 0$ centered at points of $X$.
Hence $X^{(\epsilon)}$ provides a---perhaps rough---approximation to the
topology of $\mathbb{X}$ for each $\epsilon\geq 0$, and  any discretization $0 = \epsilon_0 < \epsilon_1 < \cdots $
of $[0, \infty)$ yields an  object in $\mathbf{Top}^{\NN}$ as follows:
$\mathbb{N}\ni i \mapsto X^{(\epsilon_i)}$,
and the inclusion $X^{(\epsilon_i)}\hookrightarrow X^{(\epsilon_j)}$
is the  map associated to $i\leq j$.
Such objects allow one
to avoid   optimizing  the choice of a single $\epsilon$,
leading to several recovery theorems and algorithms for topological inference \cite{oudotBook}.
\end{example}

\begin{remark}
If  each $f_{i}: X_i \rightarrow X_{i+1}$  in  $\mathcal{X} \in \mathbf{Top}^{\NN }$ is an inclusion,
e.g. as
 in Example \ref{expl:MotivationTDA},
then $\mathcal{X}$ defines a \textit{filtered topological space}.
Specifically, the union of the $X_i$'s.
Recall that a filtered topological space consists of a  space
$X$ together with a  filtration  $X_0 \subset X_1 \subset \cdots \subset X$.
\end{remark}

This observation motivates the following definition.

\begin{definition}
We  call $\mathcal{X} \in \mathbf{Top}^{\NN} $ a \emph{filtered space}
if all its maps are inclusions.
The collection of all filtered spaces forms a full subcategory of $\textbf{Top}^{\textbf{N}}$ denoted $\textbf{FTop}$.
\end{definition}

\begin{example}\label{expl:TelescopeReplacement}
Let $\mathcal{X} \in \mathbf{Top}^\NN$,
and let $\mathcal{T}(\mathcal{X})  \in \mathbf{FTop}$
be the functor defined as follows.
For $j\in \mathbb{N}$, let
$\mathcal{T}_j(\mathcal{X})$ be the \emph{mapping
telescope} of
\[
X_0 \xrightarrow{f_0}
X_1 \xrightarrow{f_1}
\cdots
\xrightarrow{f_{j-2}}
X_{j-1}
\xrightarrow{f_{j-1}}
X_j.
\]
Explicitly,
$\mathcal{T}_j(\mathcal{X})$
is the quotient space
\[
\mathcal{T}_j(\mathcal{X})
=
X_j \times \{j\}
\sqcup
\bigg(
\bigsqcup_{i < j} X_i \times [ i, i +1]
\bigg)
\Big/ (x, i+1) \sim (f_i(x), i+1).
\]
It readily follows that
$\mathcal{T}_j(\mathcal{X})
\subset \mathcal{T}_k(\mathcal{X})$
whenever $j\leq k$,
and therefore $\mathcal{T}(\mathcal{X})$
defines an object in $\mathbf{FTop}$
called the \emph{telescope filtration} of
$\mathcal{X}$.
\end{example}

As is typical in algebraic topology,
one is interested in the interplay between diagrams
of spaces and  diagrams of algebraic objects.

\begin{example}\label{PersMod}
Let $R$ be a commutative ring with unity, and let
$\textbf{Mod}_{R}$ be the category of (left) $R$-modules and $R$-morphisms.
The typical objects in $\mathbf{Mod}_R^\mathbf{I}$
one encounters in TDA arise from objects $\mathcal{X} \in \mathbf{Top}^\mathbf{I} $
by fixing $n\in \mathbb{N}$ and taking singular homology in dimension $n$ with coefficients in $R$.
Indeed, $H_n(X_i ;R )$ is an $R$-module for each   $i \in \mathbf{I}$,
and for each morphism $i \rightarrow j$ in $\mathbf{I}$,
the  map $\mathcal{X}(i \rightarrow j) : X_i \longrightarrow X_j$
induces---in a functorial manner---a well-defined $R$-morphism from $H_n(X_i;R)$ to $H_n(X_j;R)$.
The resulting object in $\mathbf{Mod}_R^\mathbf{I}$ will  be denoted
$H_n(\mathcal{X};R)$.
\end{example}

\begin{definition}\label{def:SumIrredTensor}
Let $\mathcal{M}, \mathcal{N} \in \mathbf{Mod}^\mathbf{I}_R $,
and let
$\mathcal{M}\oplus \mathcal{N} : \mathbf{I}
\longrightarrow \mathbf{Mod}_R$ be the
functor sending  $i \in  \mathbf{I}$
to $(\mathcal{M} \oplus \mathcal{N})(i) :=
M_i \oplus N_i$,   and  each morphism
$i \rightarrow j$ in $\mathbf{I}$
to
\[
(\mathcal{M} \oplus \mathcal{N})(i \rightarrow j) := \mathcal{M}(i \rightarrow j) \oplus \mathcal{N}(i \rightarrow j).
\]
We say that $\mathcal{M} \in \mathbf{Mod}^\mathbf{I}_R $
is  \emph{indecomposable} if  $\mathcal{M} \cong \mathcal{N} \oplus \mathcal{O}$ only when either
$\mathcal{N}$ or $\mathcal{O}$ is the zero functor.
Similarly, let $\mathcal{M} \otimes \mathcal{N} : \II \longrightarrow \mathbf{Mod}_R$ be the functor sending $i \in \II$ to
$M_i \otimes_R N_i$ and $i\rightarrow j$ to $\mathcal{M}(i \rightarrow j) \otimes_R \mathcal{N}(i\rightarrow j)$,
where the tensor product $\otimes_R$  is the usual one for $R$-modules
and $R$-morphisms.
\end{definition}

\begin{example}\label{expl:interval}
Recall that an \emph{interval} in a poset $\mathcal{P}$ is a set  $  I\subset \mathcal{P}$ for which
$i,k \in I$ and $ i \preceq j \preceq k$
always imply $j\in I$.
An interval $I \subset \mathcal{P}$ defines
an object $\mathds{1}_I $ in $\mathbf{Mod}_R^\mathbf{P}$ as follows: for
 $j\preceq j'$ in $ \mathcal{P}$, let
\[
\mathds{1}_{I}(j)
:=
\left\{
  \begin{array}{cl}
    R &  \mbox{if }j\in I\\[.1cm]
    0 & \hbox{else}
  \end{array}
\right.
\hspace{1cm}
\mbox{and}
\hspace{1cm}
\mathds{1}_{I}(j\preceq j')
:=
\left\{
  \begin{array}{cl}
    id_R  &  \mbox{if }j,j'\in I\\[.1cm]
    0 & \hbox{else}
  \end{array}
\right.
\]
The  $\mathds{1}_I$'s
are called \emph{interval diagrams},
and if $\mathcal{P}$ is totally ordered, then
they are indecomposable in $\mathbf{Mod}^\mathbf{P}_R$
\cite[Lemma 4.2]{categorification}.
Moreover, if $I,J\subset \mathcal{P}$ are intervals, then
the tensor product of the corresponding interval diagrams satisfies
\[
\mathds{1}_I \otimes \mathds{1}_J \cong \mathds{1}_{I \cap J}.
\]
Indeed, $\mathds{1}_I \otimes \mathds{1}_J(r) \cong  0 $ if $r\notin I\cap J$,
and
$\mathds{1}_I \otimes \mathds{1}_J(r) \cong  R \otimes_R R \cong R$  if $r\in I\cap J$,
where the last isomorphism is given by multiplication in $R$.
Since applying  $id_R \otimes_R id_R$ and then multiplying
equals  multiplying and then applying $id_R$,
the result follows.
\end{example}

Interval diagrams are fundamental  building blocks in $\mathbf{Mod}_\F^\mathbf{P}$   \cite[Theorem 1.1]{crawley2015decomposition}:

\begin{theorem}\label{thm:CrawleyBoevey}
Let $\mathbb{F}$ be a field and
let $( \mathcal{P}, \preceq)$ be a totally
ordered set.
Suppose that $\mathcal{P}$ is separable with respect to the order topology, and
that $\mathcal{V} \in
\mathbf{Mod}_\mathbb{F}^\mathbf{P} $
satisfies $\dim_\mathbb{F}\mathcal{V}(j)<\infty $
for all $j\in \mathcal{P}$.
Then, there exists a multiset  of intervals $I \subset \mathcal{P}$  called the
\emph{barcode} of $\mathcal{V}$,  denoted $\mathsf{bcd}(\mathcal{V})$, and
so that
\[
\mathcal{V}
\cong
\bigoplus_{I \in \mathsf{bcd}(\mathcal{V})}
\mathds{1}_I.
\]
Moreover, $\mathsf{bcd}(\mathcal{V})$  only depends on---and uniquely determines---the isomorphism type of $\mathcal{V}$.

\end{theorem}

\begin{remark}
An object $\mathcal{ V } \in
 \mathbf{Mod}_\mathbb{F}^\mathbf{I} $ satisfying the condition $\dim_\mathbb{F}\mathcal{V}(j)<\infty $ \textit{for all} $j\in \mathbf{I}$ is  said to be \textit{pointwise finite}.
\end{remark}

\begin{definition}
 Let $\mathbb{F}$ and $\mathcal{P}$ be as in Theorem \ref{thm:CrawleyBoevey}, and
let $\mathcal{X} \in \mathbf{Top}^\mathbf{P}$ be so that $H_n(\mathcal{X}; \mathbb{F})$ is  pointwise finite.
The barcode of $H_n(\mathcal{X}; \mathbb{F})$, denoted
$\mathsf{bcd}_n(\mathcal{X} ;\mathbb{F})$, is the unique multiset of intervals in $\mathcal{P}$
such that
\[
H_n(\mathcal{X};\mathbb{F}) \cong
\bigoplus_{I\in \mathsf{bcd}_n(\mathcal{X}; \mathbb{F})}\mathds{1}_I.
\]
\end{definition}

It follows that $\mathsf{bcd}_n(\mathcal{X}; \mathbb{F})$ provides
a succinct description of the isomorphism type
of $H_n(\mathcal{X}; \mathbb{F})$.
The design of algorithms for the computation of barcodes, at least in the $\mathbb{N}$-indexed case, leverages a more concrete description of $H_n(\mathcal{X}; \mathbb{F})$ which we describe next.

\begin{definition}
For $\mathcal{M} =\{h_{j,i} : M_i \longrightarrow M_{j}\}_{i \preceq j \in \mathcal{P}}$  an object in $ \mathbf{Mod}_R^\mathbf{P}$,
let
\begin{equation}\label{eq:PersistenceModule}
P\mathcal{M} : = \bigoplus_{i\in \mathcal{P}} M_i.
\end{equation}
$P\mathcal{M}$ is called the \emph{persistence module} associated to $\mathcal{M}$.
\end{definition}

The  word module
stems from the following observation: if $\mathcal{ M } \in \mathbf{Mod}_R^\NN $, then $P\mathcal{M}$ is
a  graded module
over $R[t]$, the graded ring of polynomials
in a variable $t$.
Indeed,  the $R[t]$-module structure
is defined through multiplication by  $t$ as follows:
for $\mathsf{m} = (m_0,m_1,\ldots )\in P\mathcal{M}$  let
\[
t\cdot \mathsf{m} :=
\big(0 , h_0(m_0), h_1(m_1), \ldots\big)
\]
and extend the action to $R[t]$ in the usual way.
The graded nature of the multiplication
comes from noticing that
$t^k \cdot M_i \subset M_{i +k}$
for every $i,k \in \mathbb{N}$.

If $\phi: \mathcal{M} \longrightarrow \mathcal{N}$ is a morphism in
$\mathbf{Mod}_R^\NN$, then
it can be readily checked that
$P\phi := \oplus_i \phi_i : P\mathcal{M} \longrightarrow P\mathcal{N}$
is a graded $R[t]$-morphism,
and that $P$ defines a functor satisfying:

\begin{theorem}[Correspondence] \label{correspondence}
Let $\mathbf{gMod}_{R[t]}$ denote the category of graded $R[t]$-modules and graded $R[t]$-morphisms. Then,
$ P: \mathbf{Mod}_R^\NN
\longrightarrow
\mathbf{gMod}_{R[t]}$ is
an equivalence of categories.
\end{theorem}

\begin{remark}\label{rmk:interval}
 Let $ \ell < \rho $ in $\mathbb{N}\cup \{\infty\} $, and let $\mathds{1}_{[\ell,\rho)} \in  \mathbf{Mod}_R^\NN$ be the resulting interval diagram.
Then,
\[
P\mathds{1}_{[\ell,\rho)} \cong \left(t^\ell \cdot R[t]\right)/\left(t^\rho\right)
\]
as graded $R[t]$-modules, with the convention  $t^\infty = 0$.
These are called \emph{interval modules}.
\end{remark}

Let $\mathcal{X} \in \mathbf{Top}^\NN$---for instance,  encoding
multiscale approximations to the topology
of an underlying unknown space (see Example \ref{expl:MotivationTDA}).
Applying $H_n(\;\cdot\; ; R)$  as in Example \ref{PersMod} yields
an object in $\mathbf{Mod}_R^\NN$,
and applying the functor $P$
from the Correspondence Theorem (\ref{correspondence}) defines a graded $R[t]$-module.
Explicitly:

\begin{definition}
Given an object $\mathcal{ X } \in  \mathbf{Top^N}$,  its
$n$-dimensional \emph{persistent homology}
with coefficients in $R$ is the persistence module
$PH_n(\mathcal{X}; R) \in \mathbf{gMod}_{R[t]}$,
\[
PH_n(\mathcal{X};R) :=
\bigoplus_{i\in \mathbb{N}}
H_n(X_i ;R).
\]
\end{definition}

\begin{remark}
If $\mathbb{F}$ is a field and $H_n(\mathcal{X};\mathbb{F})$ is pointwise finite, then
\[
PH_n(\mathcal{X}; \mathbb{F})
\cong
\bigoplus_{[\ell,\rho) \in \mathsf{bcd}_n(\mathcal{X};\mathbb{F})}
\left(t^\ell \cdot \mathbb{F}[t]\right)
/
\left(t^\rho\right)
\] as graded $\mathbb{F}[t]$-modules.
A graded version of the structure
theorem for modules
over a PID implies that
if $PH_n(\mathcal{X}; \mathbb{F})$ is finitely generated over $\mathbb{F}[t]$, then
$\mathsf{bcd}_n(\mathcal{X};\mathbb{F})$
can be recovered from the (graded) invariant factor decomposition
of $PH_n(\mathcal{X}; \mathbb{F})$.
This is how   the first general persistent homology algorithms were implemented  \cite{CPH}.
\end{remark}

Thus far we have described persistent homology in terms of barcodes and $\F[t]$-graded modules.
The next, and final description, is in terms of the standard homology of the \emph{persistence chain complex}.
Indeed, let $\mathbf{Ch}_R$ denote the category of ($\mathbb{N}$-graded) chain
complexes of $R$-modules, and chain maps.
Recall that given two chain complexes
$C_*$
and $C_*'$,
their direct sum is given by
\[
C_*\oplus C_*'  =
\left\{\partial_i \oplus \partial'_i : C_i \oplus C_i' \longrightarrow
C_{i-1}\oplus C'_{i-1}
\right\}_{i\in \N}.
\]

\begin{definition}
Let $\mathcal{C}_* = \{f_j : C_{*j} \longrightarrow C_{*j+1}\}$
be an object in $\mathbf{Ch}_R^\NN$.
That is, each $C_{*j}$ is a chain complex
of $R$-modules, and the $f_j$'s are  chain maps.
Then, the \emph{persistence chain complex}
of $\mathcal{C}_*$ is
\[
P\mathcal{C}_* : =
\bigoplus_{j\in \mathbb{N}}
C_{*j}
\]
where each
$P\mathcal{C}_i = \bigoplus\limits_{j\in \mathbb{N}} C_{i,j} \in \mathbf{gMod}_R $,
and therefore $P\mathcal{C}_* $ is an object in the category $ \mathbf{gCh}_{R[t]}$ of chain  complexes
of graded $R[t]$-modules.
\end{definition}

\begin{remark}\label{rmk:PHCvsHPC}
Since homology commutes
with direct sums of chain complexes, then
\[
H_n(P\mathcal{C}_*)
\cong
\bigoplus_{j\in \mathbb{N}}
H_n(C_{*j})
=
PH_n(\mathcal{C}_*)
\]
as $R[t]$ modules.
\end{remark}

For  $X\in \mathbf{Top}$, let
$S_*(X;R) \in \mathbf{Ch}_R$
denote the chain complex of singular
chains in $X$ with coefficients in $R$.
Then,
 given $\mathcal{X} \in \mathbf{Top}^\NN $,
 composition of functors
yields an object
$S_*(\mathcal{X};R)$ in $\mathbf{Ch}_R^\NN$.
The associated persistence
chain complex is thus
an object $PS_*(\mathcal{X};R) $ in
$ \mathbf{gMod}_{R[t]}$
and its homology---by Remark \ref{rmk:PHCvsHPC}---recovers the
persistent homology of $\mathcal{X}$.
In other words,

\begin{theorem}[\cite{CPH}] \label{PH} Let $\mathcal{X}\in \mathbf{Top^N}$. Then, its persistent homology is isomorphic over $R[t]$ to the homology
of $PS_*(\mathcal{X},R)$:
	\begin{align}
	PH_n(\mathcal{X}; R) \cong H_n(PS_{\ast}(\mathcal{X} ; R)).
	\end{align}
\end{theorem}
For a general (small) indexing category $\mathbf{I}$
the barcode is no longer available as a discrete  invariant for pointwise finite objects in $\mathbf{Mod}_\mathbb{F}^\mathbf{I}$   \cite{multiparameter}.
However, it is always possible to consider the rank invariant:

\begin{definition}\label{def:RankInvariant}
Let $\mathcal{ M } \in \mathbf{Mod}_\mathbb{F}^\mathbf{I}$ . The \textit{rank invariant} of $\mathcal{M}$ is the
function on  morphisms of
$\II$  given by $\rho^\mathcal{ M }(i \rightarrow j) = \mathsf{rank}(\mathcal{M}(i\rightarrow j)) \in \N\cup \{\infty\} $.
For an object $\mathcal{ X }\in \mathbf{Top}^\II$
we let $\rho_n^\mathcal{ X } := \rho^{H_n(\mathcal{X};\F)}$.
\end{definition}

\begin{remark}\label{rmk:rankandbarcode}
Note that $\rho^\mathcal{ M }$ is an invariant of the isomorphism type of $\mathcal{ M }$.
It is in fact computable in polynomial time when $\mathbf{I} = \mathbf{N}^k$
and $P\mathcal{M}$ is finitely generated as an
$\F[t_1,\ldots, t_k]$ module \cite{compmultdim}.
If   $\mathcal{ M } \in \mathbf{Mod}_\F^\RR$ is pointwise finite,
then   $\rho^\mathcal{ M }$ and   $\mathsf{bcd}(\mathcal{ M })$ can be recovered from each other \cite[Theorem 12]{multiparameter}.
\end{remark}

\subsection{The Classical K\"{u}nneth Theorems}
The classical K\"{u}nneth theorem in algebraic topology relates the homology of the product of two spaces to the homology of its factors.
The relation is  via a split natural short exact sequence,
and the proof is a combination of the algebraic K\"unneth formula for chain complexes, and the Eilenberg-Zilber theorem.
Both theorems are stated next, but first here are two
relevant definitions:

\begin{definition}
	Let $R$ be a commutative ring, and let $C, C' \in \mathbf{Ch}_R$. The \textit{tensor product chain complex} $C \otimes_R C'$ consists of $R$-modules $(C \otimes_R C')_n=\bigoplus\limits_i(C_i \otimes_R C'_{n-i})$ and boundary morphisms
\[
\begin{array}{ccl}
C_i \otimes_R C_{n-i}' & \longrightarrow & (C\otimes_R C')_n \\
c \otimes_R c' &\mapsto& \partial_i c \otimes_R c'+ (-1)^i c \otimes_R \partial'_{n-i} c'
\end{array}
\]
\end{definition}
\begin{definition}
	An $R$-module $M$ is called \textit{flat}, if for every short exact sequence $0 \rightarrow A \rightarrow A' \rightarrow A'' \rightarrow 0$ of $R$-modules, the induced  sequence
\[0 \rightarrow A\otimes_R M \rightarrow A' \otimes_R M \rightarrow A'' \otimes_R M \rightarrow 0\] is also exact.
In particular, free modules are flat.
\end{definition}

\begin{theorem}[The Algebraic K\"{u}nneth Formula]\label{AKT}
	If $R$ is a PID and at least one of the chain complexes $C,C'$ is flat
(i.e., the constituent modules are flat), then for each $n\in \N$ there is a natural short exact sequence
\begin{align*}
0
\longrightarrow
\bigoplus_{i + j = n}(H_i(C) \otimes_R H_{j}(C'))
\longrightarrow
H_n(C \otimes_R C') \longrightarrow \hspace{1cm}\\
\bigoplus_{i+j} \mathsf{Tor}_R(H_i(C),H_{j-1}(C'))
\longrightarrow
0
	\end{align*}
	which splits but not naturally.
\end{theorem}
See for instance
 \cite[Chapter 5, Theorem 2.1]{hilton}.
As it is well known, when $C$ and $C'$ are the singular chain complexes of two topological spaces $X$ and $Y$, then
the Eilenberg-Zilber theorem  \cite{EZ}  provides a
link between the algebraic K\"unneth theorem and the homology of $X\times Y$:

\begin{theorem}[Eilenberg-Zilber]\label{thm:ez}
	For topological spaces $X$ and $Y$, there is a natural chain equivalence $\zeta: S_{\ast}(X;R) \otimes_R S_{\ast}(Y;R) \longrightarrow S_{\ast}(X \times Y;R)  $,
and thus
\[
H_n(X \times Y;R)
\cong
H_n\big(S_{\ast}(X;R) \otimes_R S_{\ast}(Y;R)\big)
\]
	for all $n\geq 0$.
\end{theorem}

The existence of $\zeta$ is in fact an application
of the acyclic models theorem of Eilenberg and MacLane \cite{acyclic}.
This theorem provides conditions under which two functors
to the category of chain complexes produce naturally
isomorphic homology theories.
Later on we will use   the machinery of acyclic models  to prove a persistent
Eilenberg-Zilber theorem (Theorem \ref{EZnew}), which then yields our persistent
K\"unneth formula for the tensor product (Theorem \ref{KunnethTPF}).
For now, here is the topological K\"unneth theorem:

\begin{corollary}[The topological  K\"{u}nneth formula]\label{classickunneth}
	Let $X,Y$ be topological spaces and let $R$ be a PID.
Then, there exists a natural short exact sequence
	\begin{align}\label{kunneth}
	\begin{split}
	0 \longrightarrow \bigoplus_{i+j = n}(H_i(X;R) \otimes_R H_{j}(Y;R)) \longrightarrow H_n(X \times Y;R) \longrightarrow\\ \bigoplus_{i+j = n} \mathsf{Tor}_R(H_i(X;R),H_{j-1}(Y;R)) \longrightarrow 0
	\end{split}
	\end{align} which splits, though not naturally.
\end{corollary}

\subsection{Homotopy Colimits}
The machinery of homotopy colimits will be used in Section \ref{filtrations}
to define the generalized tensor product $\otimes_\mathbf{g} : \mathbf{Top}^\NN \times \mathbf{Top}^\NN \longrightarrow
\mathbf{Top}^\NN$ appearing in  Theorem \ref{KunnethTPF}.
We provide here a basic review, but  direct the interested reader to \cite{hoco} or \cite{hocomb} for a more comprehensive presentation.
If $\mathbf{C}$ is a category, then  $\mathbf{C}^{op}$ will denote its opposite category, and $\mathcal{B}\mathbf{C}$ will denote its classifying space---i.e., the geometric realization of the nerve of $\mathbf{C}$.

\begin{definition}
The \textit{undercategory}   of $c \in \mathbf{C}$, denoted  $(c \downarrow \mathbf{C})$, is the category whose objects are pairs $(c', \sigma)$ consisting of an object $c'\in \mathbf{C}$ and a morphism $\sigma: c \rightarrow c'$ in $\mathbf{C}$. A morphism from $ (c', \sigma: c \rightarrow c')$ to $(c'', \beta: c \rightarrow c'')$ is a morphism  $\tau: c' \rightarrow c''$ in $\mathbf{C}$ such that $\tau \circ \sigma = \beta$.
\end{definition}

\begin{remark}
If $\sigma: c \rightarrow c'$ is a morphism in $\mathbf{C}$, then composing with $\sigma$ induces covariant
 functors $\sigma^* : (c' \downarrow \mathbf{C}) \rightarrow (c \downarrow \mathbf{C})$ and
 $\sigma_{op}^* : (c' \downarrow \mathbf{C})^{op} \rightarrow (c \downarrow \mathbf{C})^{op}$.
\end{remark}

\begin{definition}\label{prodcoprod}
	Let $\lbrace A_i\rbrace_{i\in I}$ be a family of objects in a category $\mathbf{C}$. The \textit{product} of this family, if it exists,
is   an object $P\in \mathbf{C}$ along with morphisms  $p_i: P \rightarrow A_i$ such that for any  $Y\in \mathbf{C}$ and any collection of morphisms $\lbrace y_i : Y \rightarrow A_i\rbrace_{i\in I}$, there exists a unique morphism $f: Y \rightarrow P$ so that $ p_i \circ f = y_i$ for each $i$.
The product, since when it exists it is unique up to isomorphism, will be denoted  $\bigtimes_{i \in A}A_i$.
The existence of $f$ for any $Y$ and any $\{y_i\}_{i\in I}$, is referred to as \emph{the universal property}
defining the categorical product.

The dual notion of \textit{coproduct} of $\lbrace A_i\rbrace_{i\in I}$ is the object $C\in\mathbf{C}$, when it exists,
along with morphisms $c_i: A_i \rightarrow C$ such that for any any object $Y$ and any collection of morphisms $\lbrace y_i : A_i \rightarrow Y\rbrace_{i\in I}$, there exists a unique morphism $g: C \rightarrow Y$ such that $g \circ c_i = y_i$ for each $i$.
The coproduct of $\{A_i\}_{i\in I}$ is denoted $\coprod_{i\in I}A_i$,
and the existence of $g$ is the universal property
defining it.
\end{definition}

\begin{example}
If  $A_i \in \textbf{Top}$, then the product $\bigtimes_{i \in A}A_i$ is the space whose underlying set is the Cartesian product endowed with the product topology, while the coproduct $\coprod_{i\in I}A_i$ is their disjoint union.
\end{example}

\begin{definition}
Let $f,g: A  \rightarrow B$ be  morphisms in a category $\mathbf{C}$.
The \emph{coequalizer} of $f$ and $g$, denoted  $\mathsf{coeq}(A  \substack{ \rightarrow \\[-0.8em] \rightarrow} B)$ if it exists, is an object $C\in \CC$ with a morphism $q: B \rightarrow C$ such that $q \circ f = q \circ g$. Moreover, if there is another pair $(C',q')$ such that $q' \circ f = q' \circ g$, then there is a unique $ u : C \rightarrow C'$ such that $u \circ q = q'$.	
\end{definition}
\begin{example}
	If $f,g: A  \rightarrow B$ are two morphisms in $\mathbf{Top}$, then $\mathsf{coeq}(A  \substack{ \rightarrow \\[-0.8em] \rightarrow} B)$ is the quotient space $B\big/ \big(f(a) \sim g(a), \;\forall a \in A\big)$.
\end{example}
\begin{definition}\label{def:coandhoco}
	Let $D: \mathbf{I} \longrightarrow \textbf{Top}$ be a functor. The \textit{colimit} of $D$ is defined as:
	\begin{align}
	\mathsf{colim}(D) := \mathsf{coeq}\left[ \coprod_{i \rightarrow j} D(i)  \substack{ \rightarrow \\[-0.8em] \rightarrow} \coprod_i D(i)  \right]
	\end{align}
	where the top map sends $D(i)$ to itself via the identity,  the bottom  map sends $D(i)$ to $D(j)$ via    $D(i \rightarrow j)$,
and the first coproduct is indexed over all morphisms in $\mathbf{I}$.
	Similarly, the \textit{homotopy colimit} of $D$ is defined as:
	\begin{align}
	\mathsf{hocolim}(D) := \mathsf{coeq}\left[ \coprod_{i \rightarrow j} D(i) \times \mathcal{ B }(j \downarrow \mathbf{I})^{op} \substack{ \rightarrow \\[-0.8em] \rightarrow} \coprod_i D(i) \times \mathcal{ B }(i \downarrow \mathbf{I})^{op} \right]
	\end{align}
	The top map   $D(i) \times \mathcal{ B }(j \downarrow \mathbf{I})^{op} \rightarrow D(j) \times \mathcal{ B }(j \downarrow \mathbf{I})^{op}$
is $D(i \rightarrow j)$ times the identity,
while the bottom map
 $D(i) \times \mathcal{ B }(j \downarrow \mathbf{I})^{op}
\rightarrow D(i) \times \mathcal{ B }(i \downarrow \mathbf{I})^{op}$
is the identity of $D(i)$ times the map on classifying spaces   induced by
$(i\rightarrow j)^*_{op} : (j\downarrow \mathbf{I})^{op} \rightarrow  (i \downarrow \mathbf{I})^{op}$.
\end{definition}

It is worth noting that if $D,D' \in  \mathbf{Top}^\II$, then any morphism  $D\rightarrow D'$
induces (functorially) a  map
$
\mathsf{hocolim}(D) \rightarrow \mathsf{hocolim}(D')
$,
and that  if
 $f  \in  \JJ^\II$ and $D\in \mathbf{Top}^\JJ$,
then $f$ induces a natural map $\phi_f : \mathsf{hocolim}(D\circ f) \rightarrow \mathsf{hocolim}(D)$
called the \emph{change of indexing category}. Moreover,  \cite{hoco}

\begin{prop}\label{prop:hocolimitcommute}
Let   $f: \mathbf{I} \rightarrow \mathbf{J}$,
$g: \mathbf{J} \rightarrow \mathbf{K}$ and $D: \mathbf{K} \rightarrow \mathbf{Top}$ be functors. Then
$\phi_{g\circ f} = \phi_g \circ \phi_f$.
\end{prop}

While  both the colimit and the homotopy colimit of a diagram of spaces  yield
methods for functorially gluing spaces along maps, only the latter  preservers  homotopy equivalences.
Indeed  \cite[Proposition 3.7]{hocomb},
\begin{theorem}\label{hocohom}
	Let $D,D'\in \mathbf{Top}^\II$.
If $D \rightarrow D'$ is a natural (weak) homotopy equivalence---i.e. each $D(i) \rightarrow D'(i)$ is a (weak) homotopy equivalence---then the induced map $\mathsf{hocolim}(D) \rightarrow \mathsf{hocolim}(D')$ is a (weak) homotopy equivalence.
\end{theorem}

When the indexing category has a \textit{terminal object}, that is, a unique object $z$ such that for all $i \in \mathbf{I}$ there is a unique morphism $i \rightarrow z$, then the homotopy colimit is particulary simple \cite[Lemma 6.8]{hoco},
\begin{theorem}\label{thm: hocoterminalobject}
	Suppose that $\mathbf{I}$ has a terminal object $z$. Then for all $D\in  \mathbf{Top}^\II$, the collapse map $\mathsf{hocolim}(D) \rightarrow D(z)$ is a weak homotopy equivalence.
\end{theorem}

Notice that by collapsing the classifying spaces
$\mathcal{B}(j \downarrow \mathbf{I})^{op}$ and $\mathcal{B}(i \downarrow \mathbf{I})^{op}$ in the definition of homotopy colimit to a point,
we recover the definition of the colimit.
Next, we will see specific conditions under which this collapse defines a homotopy equivalence
from $\mathsf{hocolim}(D)$ to $\mathsf{colim}(D)$. In order to state the result, we   recall the notion of a cofibration.

\begin{definition}
	A continuous map $f : A \rightarrow X$ is called a \textit{cofibration} if it satisfies the \textit{homotopy extension property} with respect to all spaces $Y$.
That is, given a homotopy $h_t : A \rightarrow Y$ and a map $H_0 : X \rightarrow Y$ such that $H_0\circ f = h_0 $,
then there is a homotopy $H_t : X \rightarrow Y$ such that $H_t\circ f = h_t$ for all $t$.
If $f(A)$ is a closed subset of $X $, then   $f$ is called a \textit{closed cofibration}.
\end{definition}

\begin{remark}\label{telecofibration}
The telescope filtration
$\mathcal{ T }\left(\mathcal{ X }\right) \in \mathbf{FTop}$ of $\mathcal{X} \in \mathbf{Top^N}$   satisfies that  $\mathcal{ T }_{i}(\mathcal{ X }) \hookrightarrow\mathcal{ T }_{j}(\mathcal{ X })$ is a closed cofibration for all $i\leq j$ \cite[Chapter VII, Theorem 1.5]{bredon}.
\end{remark}

Here is one situation where $\mathsf{colim}$ and $\mathsf{hocolim}$ provide similar answers,

\begin{lemma}[Projection Lemma]\label{projectionlemma}
Let $\mathcal{P}$ be a partially ordered set and $\mathbf{P}$ its poset category.
	Let $D \in  \mathbf{Top^P}$ be   such that $D(i\preceq j) : D(i) \rightarrow D(j)$ is a closed cofibration for each pair $i\preceq j$ in $\mathcal{P}$. Then the collapse map
	\begin{align*}
	\mathsf{hocolim}(D) \rightarrow \mathsf{colim}(D)
	\end{align*}
	is a homotopy equivalence.
\end{lemma}
\begin{proof}
See \cite[Proposition 3.1]{hocomb}.
\end{proof}

\section{Products of Diagrams of Spaces}\label{filtrations}

With the necessary background  in place, we now move towards the main results of the paper.
Given a category $\mathbf{S}$ of spaces (e.g., topological, metric, simplicial, etc) and a small indexing category $\mathbf{I}$,
our first objective is to identify relevant products
between $\mathbf{I}$-indexed diagrams $\mathcal{X},\mathcal{Y} \in \mathbf{S}^\mathbf{I}$.
For a particular product, the
goal  is to describe its persistent homology
in terms of the persistent homologies of $\mathcal{X}$ and $\mathcal{Y}$. This is what we call a \emph{persistent K\"unneth formula}.
Let us begin with the first product construction: the categorical product in $\mathbf{S^I}$.

\begin{definition}\label{def:catprod}
Let $\mathbf{S}$ be a category having all pairwise products (e.g., finitely complete), let
$\mathcal{X},\mathcal{Y} \in \mathbf{S^I}$, and let $\mathcal{X}\times \mathcal{Y} : \mathbf{I} \longrightarrow \mathbf{S}$
be the functor taking each object $i\in \mathbf{I}$  to the
(categorical)  product $X_i \times Y_i \in \mathbf{S}$,
and each morphism
$i\rightarrow j$ in $\mathbf{I}$   to   the unique morphism
$\mathcal{X}\times \mathcal{Y}(i\rightarrow j)$
making
the following diagram commute:
\begin{equation}\label{eq:productMorphism}
\begin{tikzcd}
&  X_i\times Y_i
\arrow[rd,"\mathcal{Y}(i \rightarrow j) \circ p_i^Y"]
\arrow[ld,"\mathcal{X}(i \rightarrow j) \circ p_i^X"']
\arrow[dashrightarrow]{d}
&  \\
X_j & X_j \times Y_j\arrow[r, "p_j^Y"']\arrow[l,"p_j^X"] &  Y_j
\end{tikzcd}
\end{equation}
The existence of $\mathcal{X}\times \mathcal{Y}(i\rightarrow j)$ follows from the universal property defining
$X_j \times Y_j$.
\end{definition}
We have the following observation:
\begin{prop}\label{catprod}
Let $\mathbf{S}$ be a category with all pairwise products.
Then, $\mathcal{X}\times \mathcal{Y} \in \mathbf{S^I}$ is the categorical product of $\mathcal{X},\mathcal{Y} \in \mathbf{S^I}$.
\end{prop}
\begin{proof}
First, note that the existence of
$p^\mathcal{X} : \mathcal{X}\times \mathcal{Y} \longrightarrow \mathcal{X}$ and
$p^\mathcal{Y} : \mathcal{X}\times \mathcal{Y} \longrightarrow \mathcal{Y}$
follows from that of $p_i^X : X_i \times Y_i \longrightarrow X_i$ and $p_i^Y : X_i \times Y_i \longrightarrow Y_i$
for each $i\in \mathbf{I}$, and the commutativity
of (\ref{eq:productMorphism}) for each $i\rightarrow j$.

Let $\mathcal{Z}\in \mathbf{S^I}$, and let
$\mu: \mathcal{Z} \rightarrow \mathcal{X}$, $\nu : \mathcal{Z}  \rightarrow \mathcal{Y}$ be morphisms.
For each $i\in \mathbf{I}$,
let
$\mu_i\times \nu_i :Z_i \rightarrow X_i\times Y_i
$ be
the unique morphism   so that
$p_i^X\circ (\mu_i\times \nu_i) =  \mu_i$ and
$p_i^Y\circ (\mu_i\times \nu_i) =  \nu_i$.
It readily follows that
$\mu\times \nu : \mathcal{Z} \longrightarrow \mathcal{X}\times \mathcal{Y} $, given by
$(\mu\times \nu)(i) = \mu_i\times \nu_i$,
is the unique morphism of $\mathbf{I}$-indexed diagrams
such that $p^\mathcal{X} \circ (\mu \times \nu) = \mu$,
and $p^\mathcal{Y} \circ (\mu \times \nu) = \nu$.
\end{proof}

Let us describe in more detail the categories we have in mind for
$\mathbf{S}$, as well as what the categorical products are in each case.

\begin{example}[Topological spaces]
	Recall that $\mathbf{Top}$ denotes the category of topological spaces and continuous maps. For  $X,Y \in \mathbf{Top}$, their Cartesian product
equipped with the product topology is the categorical product of $X$ and $Y$ in $\mathbf{Top}$.
\end{example}

\begin{example}[Metric Spaces]
Let $\mathbf{Met}$  denote the category of metric spaces and
non-expansive maps. That is,
its objects are pairs $(X,d_X)$---a set and a metric--- and its morphisms are
functions $f: X \longrightarrow Y$ so that
$d_Y(f(x) , f(x')) \leq d_X(x, x')$ for all $x,x' \in X$.

\begin{definition}
Given   $(X,d_X),(Y,d_Y) \in \mathbf{Met}$, let $X\times Y$ denote the Cartesian product of the underlying sets, and let $d_{X\times Y}$ be the maximum metric on
$X\times Y$:
\[
d_{X\times Y}\big((x,y), (x',y')\big)
:=
\max\{d_X(x,x'), d_Y(y,y')\}.
\]
\end{definition}

Since the coordinate projections $p^X : X\times Y \longrightarrow X$
and  $p^Y : X\times Y \longrightarrow Y$  are non-expansive,
it readily follows that $(X\times Y , d_{X\times Y})$ is the categorical product
of $(X,d_X)$ and $(Y,d_Y)$ in $\mathbf{Met}$.
\end{example}

\begin{example}[Simplicial Complexes]\label{expl:Simp}
Recall that a simplicial complex   is a collection
$K$ of finite nonempty sets---called simplices---so that if $ \sigma \in K$ and $ \emptyset \neq \tau \subset \sigma $, then $\tau \in K$.
Let  $K^{(n)}\subset K$ be the collection of simplices of $K$ with $n+1$ elements---these are called
$n$-simplices and we  write $v\in K^{(0)}$ instead of $\{v\} \in K^{(0)}$.
A  function $f: K \longrightarrow K'$
between simplicial complexes is called a simplicial map if
$f(\{v_0 ,\ldots, v_n\}) = \{f(v_0), \ldots, f(v_n)\}$ for every $\{v_0,\ldots, v_n\} \in K$.
Let $\mathbf{Simp}$ be the category of simplicial complexes and simplicial maps.

\begin{definition}\label{def:Simp}
For   $K, K' \in \mathbf{Simp}$, let $K\times K'$ be the smallest simplicial complex containing all Cartesian products $\sigma  \times \sigma'$,
for $\sigma\in K$ and $\sigma' \in K'$.
\end{definition}

In other words, $\tau \in K\times K'$ if and only if
there exist $\sigma \in K$ and $\sigma' \in K'$ with $\tau \subset \sigma\times \sigma'$.
Let $p : \tau \rightarrow \sigma$ and $p' : \tau \rightarrow \sigma'$
be the projection maps onto the first and second coordinate, respectively.
Since $p(\tau) \subset \sigma$, then $p(\tau) \in K$, and similarly $p'(\tau) \in K'$.
Hence $p$ and $p'$ define simplicial maps
$
K \xleftarrow{p} K\times K' \xrightarrow{p'} K'
$,
and it readily follows that $K\times K'$ is the categorical product of $K$ and $K'$ in $\mathbf{Simp}$.

As mentioned in the Introduction, the \emph{Rips} complex at scale $\epsilon \in \mathbb{R}$ is a simplicial complex which can be associated to any metric space  $(X,d_X)$.
It is widely used in TDA---when $(X,d_X)$ is the observed data set---and it is defined as
\begin{equation}\label{eq:DefRips}
	R_\epsilon({X})
	:=
	\left\{
	\{x_0,\ldots, x_n\} \subset {X} : \max_{0 \leq i, j \leq n} d_X(x_i,x_j) < \epsilon
	\right\}.
\end{equation}
Notice that any non-expansive function $f: X \rightarrow Y$   between metric spaces  extends to a simplicial map $R_\epsilon(f) : R_\epsilon(X) \rightarrow R_\epsilon(Y)$, and
thus $R_\epsilon$ defines a functor from $\mathbf{Met}$ to $\mathbf{Simp}$.
This functor is in fact  compatible with the   categorical products:

\begin{lemma}\label{lemma: rips} Let $(X,d_X)$, $(Y,d_Y)$ be metric spaces, and let
$\epsilon \in \mathbb{R}$. Then,
\[
R_\epsilon(X\times Y) = R_\epsilon(X)\times R_\epsilon(Y).
\]
\end{lemma}
\begin{proof}
See the proof of Proposition 10.2 in \cite{adams}.
\end{proof}

One drawback of the  categorical product in $\mathbf{Simp}$
is that it is not well-behaved  with respect to geometric realizations.
Indeed, recall that the geometric realization of a simplicial complex $K$---denoted $|K|$---is the collection of all
functions $\varphi: K^{(0)} \rightarrow [0,1]$ so that
$\{v\in K^{(0)}: \varphi(v) \neq 0\}\in K$, and for which $\sum\limits_{v\in K^{(0)}}\varphi(v) = 1$.
Moreover, if $\varphi, \psi \in |K|$,  then
\[
d_{|K|}(\varphi,\psi) :=
\sum_{v\in K^{(0)}} |\varphi(v) - \psi(v)|
\]
defines a metric   on $|K|$.
Every simplicial map $f: K \rightarrow L$ induces
a function $|f|: |K| \rightarrow |L|$ given by
\[
|f|(\varphi)(w)
=
\sum_{\substack{v\in K^{(0)} \\ f(v) = w}} \varphi(v)
\;\;\;\;\;
\mbox{ , }
\;\;\;\;\;
|f|(\varphi)(w) = 0 \; \mbox{ if }\;w \notin f(K^{(0)})
\]
which, as one can check, is non-expansive.
In other words, geometric realization defines a functor
$|\cdot | : \mathbf{Simp} \longrightarrow \mathbf{Met} $.
To see the incompatibility of $|\cdot|$ and the product in $\mathbf{Simp}$,
let $K = L = \{0 ,1 , \{0,1\}\}$.
It follows that $|K\times L|$ is (homeomorphic to) the geometric
3-simplex
$
\left\{
(t_0, \ldots, t_3) \in \mathbb{R}^4 : t_j \geq 0,\; t_0 + \cdots + t_3 =1
\right\}
$, while $|K|\times |L|$ is the unit square $[0,1]\times [0,1]\subset \mathbb{R}^2$.
Hence,
$|K\times L|$ is in general not   equal, nor  homeomorphic to $|K|\times |L|$.
This leads one to consider other alternatives;
specifically,  the category of \emph{ordered
simplicial complexes}.
\end{example}

\begin{example}[Ordered Simplicial Complexes]\label{expl:oSimp}
If
$K^{(0)}$ comes equipped with a partial order so that each simplex  is   totally ordered,
then we say that $K$ is an ordered simplicial complex.
A simplicial map between ordered simplicial complexes is said to be order-preserving,
if it is so between 0-simplices.
Let $\mathbf{oSimp}$ denote the resulting category.
For data analysis applications, and particularly for the Rips complex construction, $\mathbf{oSimp}$ is perhaps more relevant than $\mathbf{Simp}$. Indeed,
a data set $(X,d_X)$ stored in a computer  comes  with an explicit total   order on $X$.
If $ K,K'\in \mathbf{oSimp}$, with partial orders $\preceq, \preceq'$, respectively,
and $\sigma \in K$, $\sigma' \in K'$, then the Cartesian product
$\sigma \times \sigma'$ has a partial order $\preceq^\times$
given by $(v,v') \preceq^\times (w,w')$ if and only if $v \preceq w$ and $v' \preceq' w'$.

\begin{definition}\label{def:oSimp}
For $K,K'\in \mathbf{oSimp}$,
let  $K\oslash K'$ be defined as follows: $\tau \in K \oslash K'$
if and only if there exist $\sigma \in K$ and $\sigma' \in K'$ so that
$\tau$ is a totally ordered subset of   $\sigma \times \sigma'$, with respect to  $\preceq^\times$.
\end{definition}
Just like for  simplicial complexes, the coordinate projections
$p$ and $p'$ induce order-preserving simplicial maps
$p: K\oslash K' \rightarrow K$ and $p' : K\oslash K' \rightarrow K'$.
Moreover,

\begin{lemma}\label{osimpandsimp}
Let $K$ and $K'$ be ordered simplicial complexes. Then,
\begin{enumerate}
  \item $K\oslash K'$ is their categorical product in $\mathbf{oSimp}$.
  \item $K\oslash K' \subset K\times K'$, where the right hand side is the categorical product in $\mathbf{Simp}$.
      Moreover, $|K\oslash K'|$ is a  deformation retract of $|K\times K'|$.
  \item The product map $|p|\times |p'| : |K\oslash K'| \longrightarrow |K|\times |K'|$ is a homeomorphism.
\end{enumerate}
\end{lemma}
\begin{proof}
  See \cite{ES}, specifically Definitions 8.1, 8.8 and Lemmas 8.9, 8.11.
\end{proof}
\end{example}

The above are the categories of spaces and the products we will consider moving forward.
Going back to diagrams in $\mathbf{Top      }$, recall that a CW-complex is a topological
space $X$ together with a CW-structure.
That is, a filtration $\mathcal{X} = \{X^{(k)} \}_{k\in \mathbb{N}}$ of $X$, so that $X^{(0)}$ has the discrete topology,  $X^{(k)}$ is obtained from $X^{(k-1)}$ by attaching
$k$-dimensional cells $ D^k \cong  e^k_\alpha \subset X$ via continuous  maps $\varphi_\alpha: \partial D^k \rightarrow X^{(k-1)}$, and so that the topology on $X$ coincides with the weak topology induced by
$\mathcal{X}$.
It follows that   $\mathcal{X}\in \mathbf{FTop} \subset \mathbf{Top^N}$.
In this context, a cellular map between CW-complexes---i.e. a continuous
function $f: X \rightarrow Y$ so that $f\left(X^{(k)}\right) \subset Y^{(k)}$ for all $k$---is exactly a morphism in $\mathbf{Top^N}$.
Thus, if $\mathbf{CW}$ denotes the category of locally compact CW-complexes
(we will see in a moment why this restriction) and cellular maps,
then $\mathbf{CW}$ is a full subcategory of $\mathbf{FTop}$.

If $X$ and $Y$ are CW-complexes, then their topological product $X\times Y$
can also be written as a union of cells.
Specifically, the $k$-cells of $X\times Y$ are the products $e_\alpha^i \times e^j_\beta$ with $k = i+j$, for $e_\alpha^i$ an $i$-cell of $X$, and $e^j_\beta$ a $j$-cell of $Y$.
Thus, if
\begin{equation}\label{eq:CWproduct}
(X\times Y)^{(k)} := \bigcup\limits_{i+ j = k} X^{(i)} \times Y^{(j)}
\;\;\;\; , \;\;\;\;
k \in \mathbb{N}
\end{equation}
then $\left\{(X\times Y)^{(k)}\right\}_{k\in \mathbb{N}}$ will be a CW-structure
for $X\times Y$ provided the product topology coincides with the
 weak topology induced by (\ref{eq:CWproduct}).
If either $X$ or $Y$ are locally compact, then this will be the case
(see Theorem A.6 in \cite{hatcher}).

In summary: the categorical product in $\mathbf{Top^N}$
does not  recover the usual product of CW-complexes,
which suggests the existence of other useful (non-categorial) products in
$\mathbf{FTop}$ and
$\mathbf{Top^N}$.
The product of CW-complexes suggests the following definition,

\begin{definition}
If $\mathcal{X},\mathcal{Y} \in \mathbf{FTop}$, then their tensor product is the filtered space
\[
\mathcal{X}\otimes \mathcal{Y}
:=
\left\{
\bigcup\limits_{i + j = k} X_i \times Y_j
\right\}_{k\in \mathbb{N}}.
\]

\end{definition}

The tensor product of filtered spaces  is sometimes regarded in the literature   as the de facto  product filtration;
see for instance \cite{relhomotopy, groupoid, pacific} or \cite[Chapter 27]{modhomotopy}.
For general objects in $\mathbf{Top^N}$---e.g., when the
internal maps are not  inclusions---taking the union of, say,
$X_{i+1}\times Y_{j}$ and $X_{i}\times Y_{j+1}$ does not make sense. Instead,
the corresponding operation would be  to glue these spaces along the images of the product maps
\[
X_{i+1}\times Y_j \longleftarrow X_i\times Y_j \longrightarrow X_i \times Y_{j + 1}.
\]
Here is where the machinery of homotopy colimits comes in as a means to defining
a generalized tensor product in $\mathbf{Top^N}$:

\begin{definition}\label{definition: TPF}
	For $k\in \mathbb{N}$,
	let $\triangle_k $ be the poset category of $ \{(i,j) \in \mathbb{N}^2: i + j \leq k\}$ with its usual product order, and
	for $\mathcal{X}, \mathcal{Y} \in \mathbf{Top^N}$, let
	\[
	\mathcal{X} \boxtimes  \mathcal{Y} : \NN^2 \longrightarrow \mathbf{Top}
	\]
	be the functor sending $(i,j)$ to $X_i \times Y_j$, and
	$(i,j) \leq (s,t)$ to  $\mathcal{X}(i\leq s) \times \mathcal{Y}(j \leq t)$.
	The \emph{generalized tensor product} of $\mathcal{X}$ and $\mathcal{Y}$
	is the object $\mathcal{X}\otimes_{\textbf{g}} \mathcal{Y} \in \mathbf{Top^N}$
	given by
	\begin{align}\label{eq:gentensor}
    \begin{split}
		\mathcal{X}\otimes_\textbf{g}\mathcal{Y} (k) &:=
		\mathsf{hocolim}(\mathcal{X}\boxtimes \mathcal{Y}|_{\triangle_k}) \\
		\mathcal{X}\otimes_\textbf{g}\mathcal{Y} (k\leq k')
		&:=
		\mathsf{hocolim}\left(\mathcal{X}\boxtimes \mathcal{Y}|_{\triangle_k}  \rightarrow \mathcal{X}\boxtimes \mathcal{Y}|_{\triangle_{k'}} \right)
    \end{split}
	\end{align}
	where the latter  is the continuous map associated to  the change of indexing categories induced by  the inclusion
	$\triangle_{k} \subset  \triangle_{k'}$.
\end{definition}
\begin{remark}
The definition of the generalized tensor product can be extended to diagrams indexed by other subcategories
of $\RR$.
For example, to the poset categories of $\mathbb{R}$, $\mathbb{R}_+$
(the set of non negative reals) and $\mathbb{R}_\ast$ (the set of positive reals).
In general, for   $\mathbf{I} = \NN$, $\RR$, $\mathbf{R_+}$, or $\mathbf{R_\ast}$, we let
$\triangle_r = \{(i,j) \in \mathbf{I}^2: i + j \leq r\}$ and define
$\mathcal{X}\otimes_{\mathbf{g}}\mathcal{Y} : \mathbf{I} \longrightarrow \mathbf{Top}$ as in  (\ref{eq:gentensor}).
Although $\triangle_r$  and $\mathcal{ X } \boxtimes \mathcal{ Y }$ are different for each choice of $\mathbf{I}$, we will use the same notation for aesthetic purposes.
The choice will be clear from the context   as we will specify the $\mathbf{I}$ in $\mathbf{Top^I}$.
\end{remark}

We have the following relation between $\otimes$ and $\otimes_\mathbf{g}$,

\begin{lemma}\label{lemma:tensorandgentensor} Let $\mathcal{X}, \mathcal{Y} \in \mathbf{Top^N} $, and let $\mathcal{T(X)}, \mathcal{T(Y)} \in \mathbf{FTop}$ be their telescope filtrations.
Then, $\mathcal{X}\otimes_{\mathbf{g}} \mathcal{Y}$ is naturally homotopy equivalent to $\mathcal{T(X)}\otimes \mathcal{T(Y)}$.
\end{lemma}
\begin{proof} We will show that there is a morphism $\mathcal{X}\otimes_\mathbf{g} \mathcal{Y} \longrightarrow \mathcal{T(X)}\otimes \mathcal{T(Y)}$,
so that each   map $\big(\mathcal{X}\otimes_\mathbf{g} \mathcal{Y}\big)_k \longrightarrow \big(\mathcal{T(X)}\otimes \mathcal{T(Y)}\big)_k$, $k\in \N$,
is a homotopy equivalence.
Indeed, since  $\mathcal{T}_j(\mathcal{X})$ deformation retracts onto $X_j$, then each inclusion $X_j \hookrightarrow \mathcal{T}_j(\mathcal{X})$
is a homotopy equivalence, and thus (by Theorem \ref{hocohom}) so is   the induced map
\[
\big(\mathcal{X}\otimes_\mathbf{g} \mathcal{Y}\big)_k
=
\mathsf{hocolim}\left(\mathcal{X}\boxtimes \mathcal{Y}\big|_{\triangle_k}\right)
\longrightarrow
\mathsf{hocolim}\left(\mathcal{T(X)}\boxtimes \mathcal{T(Y)}\big|_{\triangle_k}\right).
\]

Since each inclusion $\mathcal{T}_i(\mathcal{X}) \hookrightarrow \mathcal{T}_j(\mathcal{X})$ is a closed cofibration
(Remark \ref{telecofibration}), then the projection Lemma \ref{projectionlemma}
implies that the collapse map
\[
\mathsf{hocolim}\left(\mathcal{T(X)}\boxtimes \mathcal{T(Y)}\big|_{\triangle_k}\right)
\longrightarrow
\mathsf{colim}\left(\mathcal{T(X)}\boxtimes \mathcal{T(Y)}\big|_{\triangle_k}\right)
\]
is a homotopy equivalence.
Moreover, since all the maps in $\mathcal{T(X)}\boxtimes \mathcal{T(Y)}$ are inclusions,
and the colimit of such a diagram is essentially the union of the spaces,
then   there is a natural homeomorphism
\[
\mathsf{colim}\left(\mathcal{T(X)}\boxtimes \mathcal{T(Y)}\big|_{\triangle_k}\right)
\cong
\bigcup_{i + j \leq k} \mathcal{T}_i(\mathcal{X}) \times \mathcal{T}_j(\mathcal{Y})
=
\big(\mathcal{T(X)}\otimes \mathcal{T(Y)}\big)_k
\]
which completes the proof.
\end{proof}

It follows that,
\begin{corollary}
If $\mathcal{X}, \mathcal{Y} \in \mathbf{Top^N}$, then
$
PH_n(\mathcal{X}\otimes_\mathbf{g} \mathcal{Y}; \F)
\cong
PH_n\big(\mathcal{T(X)}\otimes  \mathcal{T(Y)}; \F)
$
as graded $\F[t]$-modules.
\end{corollary}
Later on---when studying the persistent singular chains of $\mathcal{T(X)}\otimes \mathcal{T(Y)}$---it will be useful
to have Proposition \ref{prop:defretract} below.
The setup is as follows:
observe that if  $\mathcal{ X } \in \mathbf{Top^N}$ and  $0 < \epsilon < 1$, then
\[
\mathcal{ T }^\epsilon_{j}\left( \mathcal{ X }\right)
:=
X_j \times [j, j + \epsilon)
\sqcup
\bigg(
\bigsqcup_{i < j} X_i \times [ i, i +1]
\bigg)
\Big/ (x, i+1) \sim (f_i(x), i+1)
\]
is an open neighborhood of $\mathcal{T}_j(\mathcal{X})$
in $\mathcal{T}_{j+1}(\mathcal{X})$,
which deformation retracts onto
$\mathcal{T}_j(\mathcal{X})$.
Indeed, any element  $  \mathbf{x} \in \mathcal{ T }^\epsilon_j\left(  \mathcal{ X }\right) $ can be written uniquely as
$\xx = (x,t)$ where either $x \in X_j$ and $t  \in [j, j + \epsilon)$,
or $x \in X_i$ for $i < j$ and $t \in [i,i+1)$.
In these coordinates
$\mathcal{T}_j(\mathcal{X}) = \{(x,t) \in \mathcal{T}_j^\epsilon(\mathcal{X}) : t \leq j\}$, and
a deformation retraction
$h^X_j: \mathcal{ T }^\epsilon_j\left(  \mathcal{ X }\right) \times [0,1] \rightarrow \mathcal{ T }_j^\epsilon\left(  \mathcal{ X }\right)$
can be  defined as
\begin{equation}\label{eq:defRetract}
h_j^X\big((x,t), \lambda \big) =
\left\{
  \begin{array}{cll}
    (x,t) & \hbox{if } & t \leq j \\[.5cm]
    \big(x, \lambda j + (1-\lambda) t\big) & \hbox{ if  } & t \geq j
  \end{array}
\right.
\end{equation}
Let $\mathcal{ T }^\epsilon\left(\mathcal{ X }\right)\in \mathbf{FTop} $
be
$  i \mapsto  \mathcal{ T }^\epsilon_i(\mathcal{ X })$
and $(i\leq j) \mapsto \left( \mathcal{T}^\epsilon_i (\mathcal{X}) \hookrightarrow  \mathcal{T}^\epsilon_j (\mathcal{X}) \right)$.
Then,
\begin{prop}\label{prop:defretract}
If $\mathcal{ X }, \mathcal{ Y } \in \mathbf{Top^N}$ and $k\in \N$, then
$\mathcal{T}^\epsilon_k\mathcal{(X)}$ formation retracts onto $\mathcal{T}_k\mathcal{(X)}$, and
$\big(\mathcal{ T }^\epsilon\left(  \mathcal{ X }\right) \otimes \mathcal{ T }^\epsilon\left( \mathcal{ Y } \right) \big)_k$
deformation retracts onto
$\big(\mathcal{ T }\left(\mathcal{ X }\right) \otimes \mathcal{ T }\left( \mathcal{ Y } \right)\big)_k $.
\end{prop}
\begin{proof}
Given $k\in \N$, our goal is to define a deformation retraction
	\begin{align*}
	H_k: \bigcup_{i+j = k} \mathcal{ T }^\epsilon_i\left(  \mathcal{ X }\right) \times \mathcal{ T }^\epsilon_j\left(  \mathcal{ Y }\right)  \times [0,1]
\longrightarrow
\bigcup_{i+j= k}  \mathcal{ T }^\epsilon_i\left(  \mathcal{ X }\right) \times \mathcal{ T }^\epsilon_j\left(  \mathcal{ Y }\right)
	\end{align*}

To this end, we subdivide $  \left(\mathcal{ T }^\epsilon\left(  \mathcal{ X }\right) \times \mathcal{ T }^\epsilon\left(  \mathcal{ Y }\right) \right)_k$
as follows: For $i+j = k$, let
\begin{align*}
	 A_k &=   \left(\mathcal{ T } \left(  \mathcal{ X }\right) \times \mathcal{ T }\left(  \mathcal{ Y }\right) \right)_k\\
	 B_{i,j} &= X_i \times [i, i+\epsilon) \times Y_j \times [j, j+\epsilon)\\
	 C_{i,j-1} &= X_i \times (i,i+ \epsilon) \times Y_{j-1} \times [j-1+\epsilon, j)  \\
	 D_{i,j-1} &= X_i \times [i+\epsilon, i+1) \times Y_{j-1} \times (j-1, j-1+\epsilon) \\
	 E_{i,j-1} &= X_i \times (i, i + \epsilon) \times Y_{j-1} \times (j-1, j-1 +\epsilon)
\end{align*}
We further write $E_{i,j-1} = E^+_{i,j-1} \cup E^-_{i,j-1}$
where
$\big((x,t), (y,s)\big) \in  E_{i,j-1}^+$  (sim. $E^-_{i,j-1}$)  if and only if $t-i \geq s+ 1 -j$ (sim. $t-i \leq  s+ 1 -j$).
Figure \ref{fig:  defretract} below visually represents the subsets defined above for $k=3$.
\begin{figure}[!htb]
		\includegraphics[width=5cm]{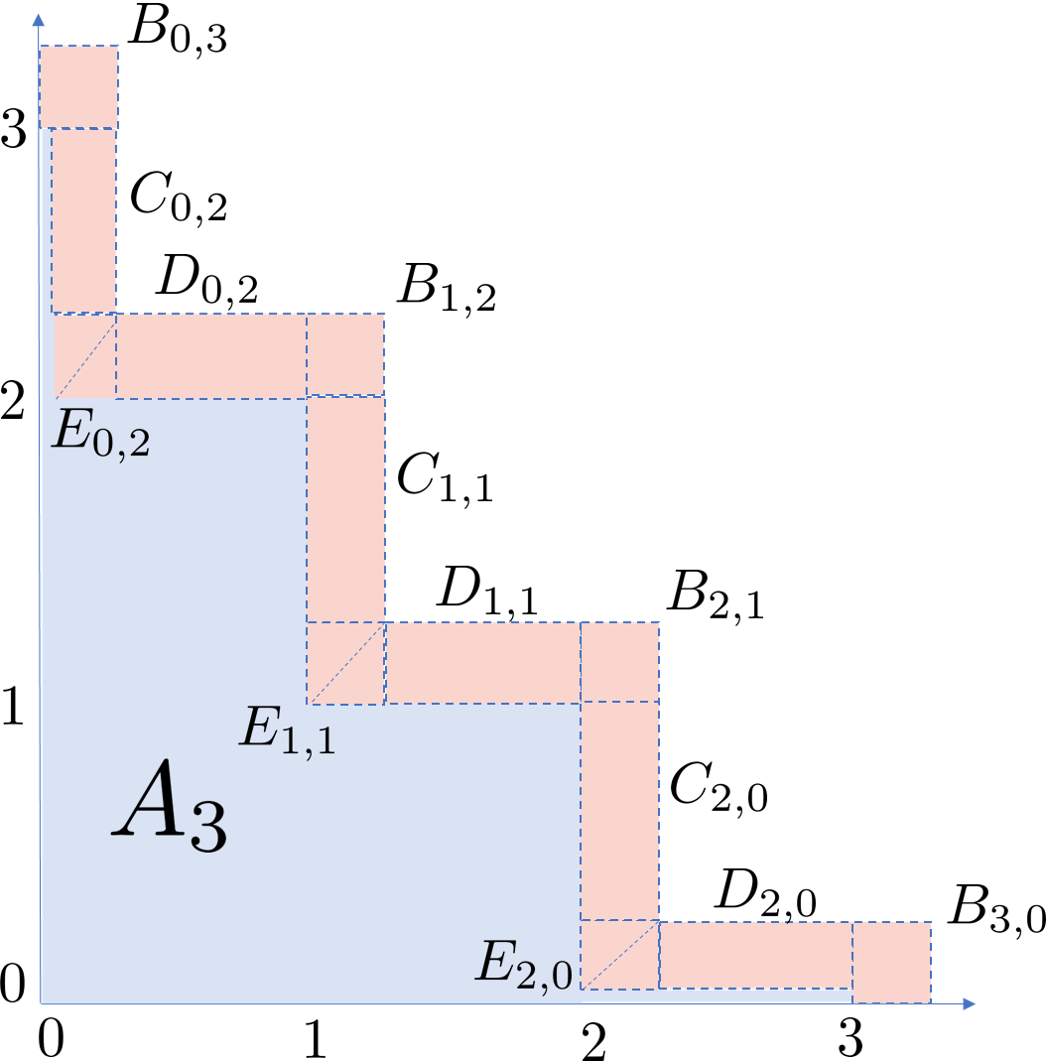}
		\caption{}
		\label{fig:  defretract}
\end{figure}

Let us now define the map $H_k$.
Let  $H_k : A_k \times [0,1] \longrightarrow A_k$ be  the projection onto the first coordinate.
On each $B_{i,j} \times [0,1]$ we  define
\[
H_k\big((\xx, \yy),\lambda\big)
=
\big(h_i^X(\xx , \lambda), h_j^Y(\yy, \lambda)\big)
\]
where $h_i^X , h_j^Y$ are given by (\ref{eq:defRetract}).
For elements in $\left(C_{i,j-1} \cup E^+_{i,j-1} \right) \times [0,1]$, $H_k$ takes
$
\big(\xx,\yy, \lambda \big)=
\big(
(x,t), (y,s), \lambda
\big)
$
to
\[
\left(
h_i^X(\xx,\lambda),
\left(
y, (1-\lambda)s + \lambda \left(j - \frac{j-s}{1+i-t}\right)
\right)
\right)
\]
	And finally, for elements in $\left(D_{i,j-1} \cup E^-_{i,j-1}\right) \times [0,1]$, $H_k$ takes $(\xx,\yy,\lambda) $ to
\[
\left(
\left(
x, (1-\lambda)t + \lambda\left(i - \frac{1+ i - t}{j-s}\right)
\right),
h_{j-1}^Y(\yy,\lambda)
\right)
\]

	Continuity readily follows by construction. Furthermore, the marginal function $H_k(-, 0)$ is the identity everywhere and $H_k( - ,1 ) \in A_k$. This finishes the proof.
\end{proof}

\subsection{A comparison theorem}\label{comparison}
Next we show that the persistent homology of the
categorical and generalized tensor products are interleaved in the log scale.
While a full characterization of   stability for each product filtration is beyond the scope of this paper,
this section showcases some of the techniques one can employ when addressing this question.

\begin{definition}
For $\delta \in \mathbb{R}$, the $\delta$-\textit{shift functor} $T_{\delta}: \mathbf{C}^\RR \rightarrow \mathbf{C}^\RR$
is defined as follows:
For $\mathcal{ M } \in\mathbf{C}^\RR $, let $T_\delta(\mathcal{ M })\in \mathbf{C^R}$
be the functor sending $r\in \mathbb{R}$ to $ \mathcal{ M }(r + \delta)$,
and   $r\leq r' \in \mathbb{R}$ to $ \mathcal{ M }(r + \delta \leq r' + \delta)$. For a morphism $\varphi: \mathcal{ M } \rightarrow \mathcal{ N }$, we get a morphism $T_\delta(\varphi): T_\delta(\mathcal{ M }) \rightarrow T_\delta(\mathcal{ N })$, and  if $\delta \geq 0$, then there is a morphism $T_{\delta}^\mathcal{ M }: \mathcal{ M } \rightarrow T_\delta(\mathcal{ M })$ satisfying
	\begin{align*}
	\left(T_{\delta}^\mathcal{ M }\right)_r = \mathcal{ M }\left( r \leq  r + \delta \right)
	\end{align*}
	for all $r \in \mathbb{R}$. We call $T_\delta^\mathcal{ M }$ the $\delta$-\textit{transition morphism}.\\
\end{definition}
The notion of interleavings,  first introduced in \cite{proximity},
allows one to compare objects in $\mathbf{Mod}_{\mathbb{F}}^\RR$.
Specifically:
\begin{definition}
Let $\mathcal{ M }, \mathcal{ N } \in \mathbf{C}^\RR$ and $\delta \geq 0$.
A $\delta$\textit{-interleaving} between $\mathcal{ M }$ and $\mathcal{ N }$
consists of
 morphisms $\varphi: \mathcal{ M } \rightarrow T_\delta(\mathcal{ N })$ and $\psi: \mathcal{ N } \rightarrow T_\delta(\mathcal{ M })$ so that $T_\delta(\varphi) \circ \psi = T_{2\delta}^\mathcal{ N }$ and $T_\delta(\psi) \circ \varphi = T_{2\delta}^\mathcal{ M }$. We define the \textit{interleaving distance} $d_I$ between $\mathcal{ M }$ and $\mathcal{ N }$ as:
	\begin{align*}
	d_I\left( \mathcal{ M }, \mathcal{ N } \right) := \inf\big\{ \delta \ \vert \; \mathcal{ M } \textrm{ and } \mathcal{N} \textrm{ are } \delta \textrm{-interleaved } \big\}.
	\end{align*}
	If there are no interleavings between $\mathcal{ M }$ and $\mathcal{ N }$, we say that $d_I(\mathcal{ M },\mathcal{ N })= \infty$.
\end{definition}

Let $\RR_\ast$ denote the poset category associated to the set of positive reals, $\mathbb{R}_\ast$, and
define the logarithm functor $\mathbf{ln}: \mathbf{Top^{R_\ast}} \rightarrow \mathbf{Top^{R}}$ as
\begin{align*}
\mathbf{ln}(\mathcal{ X })(r) &= \mathcal{X}(e^r)\\
\mathbf{ln}(\mathcal{ X })(r \leq r') &= \mathcal{ X }(e^r \leq e^{r'}).
\end{align*}
Then,
\begin{theorem}\label{thm:comparisonrstar}
If $\mathcal{ X }, \mathcal{ Y } \in \mathbf{Top^{R_\ast}}$, then
\[
d_I
\Big(
H_n\big(\mathbf{ln}( \mathcal{ X } \times \mathcal{ Y } ); \F\big)
\, ,\,
H_n\big(\mathbf{ln}( \mathcal{ X } \otimes_\mathbf{g} \mathcal{ Y }) ;\F\big)
\Big)
\;\leq\; \ln(2) .
\]
\end{theorem}
\begin{proof}
For each $r \in \mathbb{R}_\ast$,
let $\Box_r = \{(i,j) \in \mathbf{R}_*^2 :  \max\{i,j\} \leq r\}$.
We have inclusions $\triangle_{r } \hookrightarrow \Box_r \hookrightarrow \Box_{2r}$ and $\Box_r \hookrightarrow \triangle_{2r}$, and thus by Proposition \ref{prop:hocolimitcommute}, the triangles in the diagram
	\begin{equation}\label{eq:triangleCommute}
	\begin{tikzcd}[column sep= normal]
	\mathsf{hocolim}(\mathcal{ X } \boxtimes\mathcal{ Y } \vert_{\triangle_{2r }}) \ar[r] \ar[dr] & \mathsf{hocolim}(\mathcal{ X } \boxtimes\mathcal{ Y } \vert_{\triangle_{8r }}) \\
	\mathsf{hocolim}(\mathcal{ X } \boxtimes\mathcal{ Y } \vert_{\Box_r}) \ar[r] \ar[u]  & \mathsf{hocolim}(\mathcal{ X } \boxtimes\mathcal{ Y } \vert_{\Box_{4r}}) \ar[u]
	\end{tikzcd}
	\end{equation}
	commute.
Since $\Box_r$ has a terminal object, namely $(r,r)$, then the
natural map
\[
\mathsf{hocolim}(\mathcal{X}\boxtimes \mathcal{Y}\big|_{\Box_r})
\longrightarrow X_r \times Y_r
\]
is a weak homotopy equivalence (Theorem \ref{thm: hocoterminalobject}).
Such maps induce isomorphisms at the level of homology
\cite[Theorem 4.21]{hatcher},
and thus applying $\mathbf{ln}$ followed by
$H_n( \;\cdot \; ; \F)$ turns (\ref{eq:triangleCommute}) into
a $\ln(2)$-interleaving of the desired objects
in $\mathbf{Mod}_\F^\mathbf{R}$.
\end{proof}

If $\mathcal{M,N} \in \mathbf{Mod}_\F^\mathbf{R}$ are pointwise finite,
then---in addition to  their interleaving distance
$d_I(\mathcal{M,N})$---one can  compute the \emph{bottleneck distance}
$d_B$ between $\bcd(\mathcal{M})$ and $\bcd(\mathcal{N})$
as follows.
A matching between two multisets $M$ and $N$ is a multiset bijection
$\varphi: S_M \longrightarrow S_N$ between some  $S_M \subset M$ and
$S_N \subset N$.
Given $\varphi$, we say that
$I \in S_M$ and  $\varphi(I) \in S_N$ are matched,
and all other elements of $(M\smallsetminus S_M) \cup (N \smallsetminus S_N)$
are called unmatched.
Let $\delta > 0$.
A matching between $\bcd(\mathcal{M})$ and $\bcd(\mathcal{N})$
is called a $\delta$-matching if the following conditions hold:
\begin{enumerate}
  \item If $I$ and $J$ are matched, then
    \[\max\{|\ell_I - \ell_J|, |\rho_I - \rho_J|\} \leq \delta.\] Recall that $\ell_I$ and $\rho_I$ are, respectively, the
    left and right endpoints of the interval $I$.
  \item If $I$ is unmatched, then $|\ell_I - \rho_I| \leq 2\delta$.
\end{enumerate}

The bottleneck distance $d_B$ between $\bcd(\mathcal{M})$
and $\bcd(\mathcal{N})$ is defined as:
\[
d_B \big(\bcd(\mathcal{M}), \bcd(\mathcal{N})\big)
:=
\inf
\big\{
\delta \;|\;
\bcd(\mathcal{M}) \mbox{ and } \bcd(\mathcal{N})
\mbox{ are $\delta$-matched }
\big\}
\]
If no $\delta$-matchings exist, then
the bottleneck distance is $\infty$.

The \emph{Isometry Theorem} \cite{matchings}
implies  that  if $\mathcal{ M },\mathcal{ N } \in \mathbf{Mod}_{\mathbb{F}}^\RR $ are pointwise finite, then \[
d_I(\mathcal{ M }, \mathcal{ N })
=
d_B\left(\mathsf{bcd}(\mathcal{ M }), \mathsf{bcd}(\mathcal{ N })\right).
\]
The following corollary is a direct consequence
of the above equality and Theorem \ref{thm:comparisonrstar}.

\begin{corollary}\label{cor:comparisonbarcode}
Let $\mathcal{ X }, \mathcal{ Y } \in \mathbf{Top^\mathbf{R_*}}$
be so that $\mathcal{X}\times \mathcal{Y}$
and $\mathcal{X}\otimes_\mathbf{g} \mathcal{Y}$
are pointwise finite.
Then
	\begin{enumerate}
		\item If under a matching  realizing
    $d_B\big(\bcd_n(\mathcal{X}\times \mathcal{Y}),
\bcd_n(\mathcal{X}\otimes_\mathbf{g} \mathcal{Y})\big)$---which exists
by \cite[Theorem 5.12]{chazal2016structure}---we have that $I$ is matched to $J$, then
		\begin{align*}
		\frac{1}{2} \leq \frac{\ell_I}{\ell_J } , \frac{\rho_I }{\rho_J } \leq 2.
		\end{align*}
		\item If  $I$ is unmatched, then $\frac{\rho_I}{\ell_I } \leq 4.$
	\end{enumerate}
\end{corollary}

\section{A Persistent K\"{u}nneth formula for the categorical product} \label{Dthm}
Let $\mathbf{I}$ be a small indexing category and let $\mathbf{S}$ be  any of the categories of spaces
 from Section \ref{filtrations} (that is $\mathbf{S}$ = $\mathbf{Top}$, $\mathbf{Met}$, $\mathbf{Simp}$, or $\mathbf{oSimp}$).
In this section we relate the rank invariants of two objects in $\mathbf{S^I}$ to that of their categorical product,
and  prove the persistent K\"{u}nneth formula for the categorical product  in $\mathbf{S^P}$,
where $\mathbf{P}$ is the poset category of a separable totally ordered poset $\mathcal{P}$.
\begin{remark}\label{rmk:spaceskunneth}
	Recall that the topological K\"{u}nneth formula (Corollary \ref{classickunneth}) holds for objects $X, Y\in \mathbf{Top}$ and singular homology.
Every $ (X, d_X) \in \mathbf{Met}$ can be viewed as a topological space via the metric topology,
and since the maximum metric induces the product topology,
then
the K\"{u}nneth formula holds for  $ \mathbf{Met}$.
Let $K, L \in \mathbf{oSimp}$, and let $\vert \cdot \vert : \textbf{oSimp} \rightarrow \textbf{Top}$ be the geometric realization functor.
From the homeomorphism $ \vert K \oslash L\vert \rightarrow \vert K \vert \times \vert L \vert$ and the natural isomorphism
$H_n(|K| ;R) \cong H_n(K;R)$
between singular and simplicial homology, the formula holds in $ \mathbf{oSimp}$. Using (2) of lemma \ref{osimpandsimp}, the formula also holds for $ \mathbf{Simp}$.
\end{remark}
Let $\mathbb{F}$ be a field and let $\mathcal{ X }$ and $\mathcal{ Y }$ be objects in $\mathbf{S^I}$. Using the topological K\"{u}nneth formula   and the observation that $H_n(X_r ; \mathbb{F})$ and $H_n(Y_r ; \mathbb{F})$ are vector spaces for each $n \in \mathbb{N}$, $r \in \mathbf{I}$, we obtain  isomorphisms
\begin{align*}
\bigoplus_{i+j = n} H_i(X_r ; \F) \otimes_{\F} H_j(Y_r; \F) \xrightarrow{\;\;\cong\;\;} H_n(X_r \times Y_r; \F).
\end{align*}
Using  naturality, this yields an isomorphism in $\mathbf{Mod^I_\mathbb{F}}$ between $H_n( \mathcal{ X } \times \mathcal{ Y }; \F)$ and $\mathcal{ M }_n$, where
\begin{equation}\label{eq:defModuleM}
\mathcal{ M }_n(r) := \bigoplus_{i+j = n} H_i(X_r ; \F) \otimes_{\mathbb{F}} H_j(Y_r; \F)
\end{equation}
and the homomorphism $\mathcal{ M }_n(r \rightarrow r') $
is the sum of the ones induced between tensor products
by the maps $X_r \rightarrow X_{r'}$ and $Y_r \rightarrow Y_{r'}$.
In other words, and using the notation of tensor product and direct sum
of objects in $\mathbf{Mod}_\F^\II$ (see Definition \ref{def:SumIrredTensor}),
we obtain the following:
\begin{lemma}\label{lemma:KunnethCategorical}
 If $\mathcal{X,Y} \in \mathbf{S^I}$, then for every $n\in \N$ and every field
$\F$
\[
H_n(\mathcal{X\times Y}; \F) \cong \bigoplus_{i+j =n}
H_i(\mathcal{X};\F) \otimes H_j(\mathcal{Y};\F)
\] in $\mathbf{Mod}^\II_\F$.
\end{lemma}

An immediate consequence is the following calculation at the level
of rank invariants (see Definition \ref{def:RankInvariant}):
\begin{prop}\label{prop:rank}
Let $\mathcal{X,Y} \in \mathbf{S^I}$ and let
$\rho^{\mathcal{ X }}$, $\rho^{\mathcal{ Y }}$ be their rank invariants. Then
	\begin{align*}
	\rho_n^{\mathcal{ X} \times \mathcal{ Y }}(r\rightarrow r') = \sum_{i+j = n} \rho_i^{\mathcal{ X}}(r\rightarrow r') \cdot \rho_j^{\mathcal{ Y}}(r \rightarrow r')
	\end{align*}
for every $r\rightarrow r'$ in $\II$,	with the convention that $\infty \cdot 0 = 0 = 0 \cdot \infty$.
\end{prop}
\begin{proof}
This follows by direct inspection of the homomorphism
\begin{align*}
f_{r',r} \otimes_{\mathbb{F}} g_{r',r} :
H_i(X_r;\F) \otimes_{\mathbb{F}} H_j(Y_r;\F) &\longrightarrow
H_i(X_{r'};\F) \otimes_{\mathbb{F}} H_j(Y_{r'};\F)\\
a \otimes_{\mathbb{F}} b &\mapsto f_{r',r} (a) \otimes_{\mathbb{F}} g_{r',r}(b)
\end{align*}
for each $r\rightarrow r' \in \mathbf{I}$.
Indeed, since $\mathsf{Img}(f_{r',r} \otimes_{\mathbb{F}} g_{r',r}) = \mathsf{Img} (f_{r',r}) \otimes_\F \mathsf{Img}(g_{r',r})$,
then the result follows from the fact that
 the dimension of the tensor product of two vector spaces equals
 the product of the dimensions.
\end{proof}
We are now ready to prove the main result of this section: the K\"unneth formula for the categorical product $\mathcal{ X } \times \mathcal{ Y }$.
\begin{theorem}\label{barcodeDPF}
Let $\mathbf{P}$ be the poset category of a separable (with respect to the order topology) totally ordered set.
Let
 $\mathcal{X},\mathcal{Y} \in \mathbf{S^P}$
be $\mathbf{P}$-indexed diagrams of spaces,
and  assume that $H_i(\mathcal{ X }; \F)$ and $H_j(\mathcal{ Y }; \F)$ are pointwise finite for each $0 \leq i,j \leq n$.
Then  $H_n(\mathcal{X}\times \mathcal{Y}; \F)$ is pointwise finite, and its
barcode satisfies:
\[
\bcd_n\left( \mathcal{X} \times \mathcal{Y}  ; \F\right) =
\bigcup_{i+j=n}
\Big\{
I \cap J   \; \Big\vert \; I \in \mathsf{bcd}_{i}(\mathcal{X};\F), \;J \in \mathsf{bcd}_{j}(\mathcal{Y};\F)
\Big\}
\] where the union on the right is of multisets (i.e., repetitions may occur).
\end{theorem}

\begin{proof}
The fact that $H_n(\mathcal{X\times Y}; \F)$ is pointwise finite follows  directly from Lemma \ref{lemma:KunnethCategorical}.
As for the barcode formula, and removing the field $\F$ from the notation,
we have the following sequence of isomorphisms
\begin{eqnarray*}
H_n(\mathcal{X\times Y})
&\cong&
\bigoplus_{i+j = n} H_i(\mathcal{X})\otimes H_j(\mathcal{Y}) \\
&\cong&
\bigoplus_{i+j = n}
\left(\bigoplus_{I \in \bcd_i(\mathcal{X})}  \mathds{1}_I\right)
\otimes
\left(\bigoplus_{J \in \bcd_j(\mathcal{Y})}  \mathds{1}_J\right)\\[.2cm]
&\cong&
\bigoplus_{i+j = n} \;
\bigoplus_{\substack{I \in \bcd_i(\mathcal{X}) \\ J \in \bcd_j(\mathcal{Y})}}
\mathds{1}_I \otimes \mathds{1}_J \\[.2cm]
&\cong&
\bigoplus_{i+j = n} \;
\bigoplus_{\substack{I \in \bcd_i(\mathcal{X}) \\ J \in \bcd_j(\mathcal{Y})}}
\mathds{1}_{I\cap J}.
\end{eqnarray*}
The direct sum distributes with respect to the
tensor product in $\mathbf{Mod}_\F^\mathbf{P}$ since it does so in $\mathbf{Mod}_\F$.
The last isomorphism
is the one from Example \ref{expl:interval}, and
the theorem follows from the uniqueness of the barcode.
\end{proof}
This result generalizes to the product of $k$ objects in $\mathbf{S^P}$ by inductively using $\mathsf{bcd}_i\left( \mathcal{ X }_1 \times \cdots \times \mathcal{ X }_{k-1} \right)$ and $\mathsf{bcd}_j\left( \mathcal{ X }_k \right)$ to compute $\mathsf{bcd}_{i+j}\left( \mathcal{ X }_1 \times \cdots \times \mathcal{ X }_k \right)$:
\begin{corollary}
	Let $\mathcal{ X }_1, \dots, \mathcal{ X }_k \in \mathbf{S^P}$. Assume that for each $1 \leq j \leq k$, $0 \leq n_j \leq n$, $H_{n_j}(\mathcal{ X }_j)$ is pointwise finite. Then
	\begin{align*}
	\mathsf{bcd}_n\left( \mathcal{ X }_1 \times \cdots \times \mathcal{ X }_k \right)= \bigg\lbrace I_1 \cap \cdots \cap  I_k \ \bigg\vert \ I_j \in  \mathsf{bcd}_{n_j}(\mathcal{ X }_j),\;\; \sum_{j=1}^{k} n_j = n\bigg\rbrace.
	\end{align*}
\end{corollary}

The result which is perhaps most relevant
to persistent homology computations with Rips complexes is as follows:
\begin{corollary}\label{cor:RipsKunneth}
Let $(X,d_X), (Y,d_Y)$ be finite metric spaces
and let $\bcd_n^\mathcal{R}(X,d_X)$ be the barcode of the Rips filtration
$\mathcal{R}(X, d_X) := \{R_\epsilon(X,d_X)\}_{\epsilon \geq 0}$. Then,
\[
\bcd_n^\mathcal{R}(X\times Y,d_{X\times Y}) = \bigcup_{i + j = n} \Big\{ I \cap J \; \Big| \; I \in \bcd^\mathcal{R}_i(X,d_X) \, , \; J \in \bcd_j^\mathcal{R}(Y,d_Y)\Big\}
\]
for all $n\in\N$, if $d_{X\times Y}$ is the maximum metric.
\end{corollary}
\begin{proof}
Fix total orders on  $X$ and $Y$.
Then,  their $\epsilon$-Rips complexes
are ordered simplicial complexes, and
by lemmas \ref{lemma: rips} and \ref{osimpandsimp} we have that:
\[
|R_{\epsilon}(X \times Y)|
=
|R_{\epsilon}({X}) \times R_{\epsilon}({Y}) |
\simeq
|R_{\epsilon}({X}) \oslash R_{\epsilon}({Y}) |
\cong
|R_\epsilon(X)|\times |R_\epsilon(Y)|.
\]
Since the equivalences commute with inclusions of Rips complexes, the result follows.
\end{proof}

\section{A Persistent K\"{u}nneth formula for the generalized tensor product} \label{TPthm}
We will now establish  the K\"{u}nneth formula for the
generalized
tensor product $\mathcal{X}\otimes_\mathbf{g} \mathcal{Y}$ of two objects  $\mathcal{X,Y}\in\textbf{Top}^{\textbf{N}}$.
We start with a brief discussion on the grading of the tensor product
and Tor of two graded modules over a graded ring $R$.
The discussion then specializes  to  $R = \F[t]$ in order
to establish the $\F[t]$-isomorphisms
\begin{align}\label{eq:TensorTorCalculation}
\begin{split}
\left(
P\mathds{1}_{I = [\ell_I , \rho_I)}
\right)
\otimes_{\F[t]}
\left(
P\mathds{1}_{J = [\ell_J, \rho_J)}
\right)
&\cong
P\mathds{1}_{(\ell_J + I )\cap (\ell_I + J)}  \\[.2cm]
\mathsf{Tor}_{\F[t]}
\left(
P\mathds{1}_{I = [\ell_I , \rho_I)}
\; , \;
P\mathds{1}_{J = [\ell_J, \rho_J)}
\right)
&\cong
P\mathds{1}_{(\rho_J + I )\cap (\rho_I + J)}
\end{split}
\end{align}
for  $I,J \subset \N$. We end
with a persistent Eilenberg-Zilber theorem, from which the aforementioned persistent
K\"unneth theorem will follow.

\subsection{The tensor product and Tor of graded modules}\label{tensor}
Recall that a commutative ring $R$ with unity is called graded if it can be written
as $R = R_0 \oplus R_1 \oplus \cdots $ where each $R_i$ is an additive
subgroup of $R$ and $R_i R_j \subset R_{i+j}$ for every $i,j \in \N$.
Similarly, an $R$-module $A$ is   graded if it can be written
as $A = A_0 \oplus A_1 \oplus \cdots$, where  the $A_i$'s are subgroups of
$A$ and $R_i A_j \subset A_{i+j}$ for all $i,j \in \N$.
An element $a\in A$ (resp. in  $R$) is called homogeneous of degree $i\in \N$
if $a \in A_i$ (resp. in $R_i$), and we will use the shorthand  $\mathsf{deg}(a) = i$.

Let $A,B$ be graded modules over a graded ring $R$.
Our first goal is to describe a
direct sum decomposition of
the $R$-module
$A\otimes_R B$
 inducing the structure of a graded $R$-module.
Indeed, since  $A_i$ and  $B_j$ are $R_0$-modules for all $i,j \in \N$, then so are
$A,B$ and   we have the $R_0$-isomorphism
\[
A \otimes_{R_0} B  \cong
\bigoplus_{k\in \N}
\left(
\bigoplus_{i+j = k}
A_i \otimes_{R_0} B_j
\right).
\]
Fix $k\in \N$ and let
$J_k$ be the $R_0$-submodule of $A\otimes_{R_0} B$ generated by elements
of the form $(ra)\otimes b - a \otimes (r b)$, where
$a\in A$, $b\in B$ and $r \in R$ are homogeneous elements with
$\mathsf{deg}(a) + \mathsf{deg}(r) + \mathsf{deg}(b) = k $.
It follows that $J_k$ is a submodule of
\[
\left(A \otimes_{R_0} B\right)_k
:=
\bigoplus_{i+j = k}
A_i \otimes_{R_0} B_j
\]
and thus $J_k \cap J_\ell = \{0\}$ if $k \neq \ell$.
Let $J = J_0 \oplus J_1 \oplus \cdots$.
Any homogeneous element  $r\in R$ of degree $\ell$ induces a well-defined $R_0$-homomorphism
\[
\begin{array}{ccl}
(A\otimes_{R_0} B)_k  / J_k  &\longrightarrow & (A\otimes_{R_0} B)_{k+\ell}  / J_{k+\ell} \\[.2cm]
a\otimes b \,+\, J_k &\mapsto& (r a)\otimes b\, +\, J_{k + \ell}
\,=\, a \otimes (rb) \,+\, J_{k + \ell}
\end{array}
\]
and therefore
\begin{equation}\label{eq:TensorGrading}
A \otimes_{R_0} B/J \cong
\bigoplus_{k\in \N} (A\otimes_{R_0} B)_k / J_k
\end{equation}
inherits the structure of  a graded $R$-module.

The inclusion $R_0 \hookrightarrow R$ induces an
$R_0$-epimorphism $\iota: A \otimes_{R_0} B \longrightarrow A \otimes_R B$
with $ \mathsf{ker}(\iota) = J$.
Indeed, the elements of $J$ are exactly the relations
missing between $A \otimes_{R_0} B$ and $A \otimes_R B$ when defining
the latter as a quotient module.
It follows that $\iota$ induces an isomorphism
\[
\iota_* : A\otimes_{R_0} B /J \longrightarrow A\otimes_{R} B
\]of $R$-modules,
and the grading on the codomain induced by (\ref{eq:TensorGrading}) turns
$\iota_*$ into a graded isomorphism of graded $R$-modules.
We summarize this analysis as follows:

\begin{prop}\label{gradedtensor}
Let $A,B$ be graded modules over a graded ring $R$.
Then $A\otimes_R B $ decomposes as  the direct sum of the subgroups
\[
(A\otimes_R B)_k =
\left\{
\sum_\alpha a_\alpha \otimes b_\alpha \; \Big| \;
a_\alpha\in A, b_\alpha\in B  \mbox{ homogeneous, }
\mathsf{deg}(a_\alpha) + \mathsf{deg}(b_\alpha) = k
\right\}
\] and $R_\ell \cdot (A\otimes_R B)_k \subset (A\otimes_R B)_{k + \ell}$
for all $k,\ell \in \N$. In other words,
 $A\otimes_R B$ is a graded $R$-module.
\end{prop}

And now we prove the proposition for $\mathsf{Tor}_R(A,B)$:
\begin{prop}
If $A,B$ are graded $R$-modules,
then so is $\mathsf{Tor}_R(A,B)$.
\end{prop}
\begin{proof}
Fix a graded free resolution
\[
\cdots
\rightarrow
\mathcal{F}_2(A)
\rightarrow
\mathcal{F}_1(A)
\rightarrow
\mathcal{F}_0(A)
\rightarrow \mathcal{F}_{-1}(A)= A
\rightarrow 0
\]
defined inductively as follows.
Let $\mathcal{F}_0(A)$ be the free $R$-module generated
by the homogeneous elements of $A$,
and define  for each $i\in \N$ the additive subgroup
\[
\mathcal{F}_{0,i}(A)
=
\left\{
\sum_\alpha r_\alpha \cdot a_\alpha \; \Big| \;
 r_\alpha \in R, a_\alpha \in A\mbox{ homogeneous with } \mathsf{deg}(r_\alpha) + \mathsf{deg}(a_\alpha) = i
\right\}.
\]
Then
$\mathcal{F}_0(A) = \mathcal{F}_{0,0}(A) \oplus \mathcal{F}_{0,1}(A)
\oplus \cdots$,
and  with this decomposition,
the natural map $\mathcal{F}_0(A) \rightarrow \mathcal{F}_{-1}(A)$
is a surjective graded homomorphism of graded $R$-modules.
In particular, the kernel of this homomorphism is a graded
$R$-submodule of $\mathcal{F}_0(A)$.
We proceed inductively by letting
$\mathcal{F}_{j+1}(A)$ be the  graded free $R$-module
generated by the homogeneous elements in the kernel
of $\mathcal{F}_j(A) \rightarrow \mathcal{F}_{j-1}(A)$.

Taking the tensor product with $B$ over $R$ yields
\[
\cdots
\rightarrow
\mathcal{F}_2(A) \otimes_R B
\rightarrow
\mathcal{F}_1(A) \otimes_R B
\rightarrow
\mathcal{F}_0(A) \otimes_R B
\rightarrow A \otimes_R B
\rightarrow 0
\]
and we note that:
\begin{enumerate}
    \item  Each $\mathcal{F}_j(A)\otimes_R B$ is a graded $R$-module, by Proposition \ref{gradedtensor},
and each $\mathcal{F}_j(A)\otimes_R B \rightarrow \mathcal{F}_{j-1}(A)\otimes_R B$ is a graded $R$-homomorphism.
    \item $\mathsf{ker}\Big( \mathcal{F}_j(A)\otimes_R B \rightarrow \mathcal{F}_{j-1}(A)\otimes_R B \Big) $ decomposes as the direct sum
\[
\bigoplus\limits_{k\in \N}
\mathsf{ker}\Big(
\big(\mathcal{F}_j(A)\otimes_R B\big)_k
\rightarrow
\big(\mathcal{F}_{j-1}(A)\otimes_R B\big)_k
\Big)
\]
and similarly, $\mathsf{Img}\Big( \mathcal{F}_{j+1}(A)\otimes_R B \rightarrow \mathcal{F}_{j}(A)\otimes_R B \Big) $ decomposes as
\[
\bigoplus\limits_{k\in \N}
\mathsf{Img}\Big(
\big(\mathcal{F}_{j+1}(A)\otimes_R B\big)_k
\rightarrow
\big(\mathcal{F}_{j}(A)\otimes_R B\big)_k
\Big)
\] with the image being a graded $R$-submodule of the kernel.
    \item $\mathsf{Tor}_R(A,B)$ inherits the structure of a graded $R$-module
from the decomposition
\[
\mathsf{Tor}_R(A,B) \cong
\bigoplus\limits_{k\in \N}\frac
{\mathsf{ker}\Big(
\big(\mathcal{F}_1(A)\otimes_R B\big)_k
\rightarrow
\big(\mathcal{F}_{0}(A)\otimes_R B\big)_k
\Big)}
{\mathsf{Img}\Big(
\big(\mathcal{F}_{2}(A)\otimes_R B\big)_k
\rightarrow
\big(\mathcal{F}_{1}(A)\otimes_R B\big)_k
\Big)}
\]
	\end{enumerate} finishing the proof.
\end{proof}

We now specialize to the case $R = \F[t]$ and present explicit
computations for interval modules.
For $m \in \mathbb{N},k  \in \mathbb{N}\cup \{\infty\}$,
let $\mathcal{ I }_{m,k}$ denote the graded $\F[t]$-module $P\mathds{1}_{[m,m+k)}$.
For brevity, we will drop the subscript $k$ if it is $\infty$.
That is, we will use $\mathcal{ I }_m$ to denote $P\mathds{1}_{[m,\infty)}$.
Then,
$\mathcal{ I }_{m,k} \cong  t^m\mathbb{F}[t]\big/(t^{m+k})$ and
$\mathcal{ I }_{m} \cong t^m\mathbb{F}[t]$.
\begin{prop}\label{intervaltensor}
	Let $m ,n \in \mathbb{N}$, $k, l \in \mathbb{N} \cup \{\infty\}$. Then $\mathcal{ I }_{m,k} \otimes_{\mathbb{F}[t]} \mathcal{ I }_{n,l} \cong \mathcal{ I }_{m+n,\min\{k,l\}}$ and $\mathsf{Tor}_{\mathbb{F}[t]}(\mathcal{ I }_{m,k}, \mathcal{ I }_{n,l}) \cong \mathcal{ I }_{m+n+\max\{k,l\},\min\{k,l\}} $.
\end{prop}
\begin{proof}
	There is a graded  $\F[t]$-isomorphism $\mathcal{ I }_{m} \otimes_{\mathbb{F}[t]}	\mathcal{ I }_{n} \rightarrow \mathcal{ I }_{m+n}$, which takes the generator $t^m \otimes_{\mathbb{F}[t]} t^n $ to $ t^{m+n}$. Since $\mathcal{ I }_{m,k}$ and $\mathcal{ I }_{n,l}$ are quotient modules of $\mathcal{ I }_{m}$ and $	\mathcal{ I }_{n}$ respectively, their tensor product is isomorphic to a quotient module of $\mathcal{ I }_{m+n}$, say $\mathcal{ I }_{m+n,q}$ for some appropriate $1 \leq q \leq \infty$. The isomorphism takes $t^m \otimes_{\mathbb{F}[t]} t^n$ to $ t^{m+n}$ and the action of $t^p$ gives $t^p\cdot \left( t^m \otimes_{\mathbb{F}[t]} t^n \right) \rightarrow t^{m+n+p}$.	The left hand side is zero if either $t^{p+m}$ is zero in $\mathcal{ I }_{m,k}$ or if $t^{p+n}$ is zero in $\mathcal{ I }_{n,l}$. This implies that if $p \geq \min\lbrace k,l \rbrace$, then $t^{m+n+p}$ is zero in $\mathcal{ I }_{m+n,q}$. The minimum of all such $p$ is $\min\lbrace k,l \rbrace$ and therefore, $q=\min\{k,l\}$.\\
	To compute $\mathsf{Tor}_{\mathbb{F}[t]}\left(\mathcal{ I }_{m,k}, \mathcal{ I }_{n,l}\right)$, we fix  a  graded free resolution for $\mathcal{ I }_{m,k}$:
	\begin{align*}
	0 \rightarrow \mathcal{ I }_{m+k} \rightarrow \mathcal{ I }_{m}  \rightarrow \mathcal{ I }_{m,k}\rightarrow 0 .
	\end{align*}
	Taking the tensor product of the sequence with $\mathcal{ I }_{n,l}$ gives us:
	\begin{align*}
	\mathcal{ I }_{m+k}\otimes_{\mathbb{F}[t]} \mathcal{ I }_{n,l} \rightarrow  \mathcal{ I }_{m} \otimes_{\mathbb{F}[t]} \mathcal{ I }_{n,l} \rightarrow \mathcal{ I }_{m,k} \otimes_{\mathbb{F}[t]} \mathcal{ I }_{n,l} \rightarrow 0.
	\end{align*}
	Using the definition of $\mathsf{Tor}$ and the tensor product of intervals, we get that
	\begin{align*}
	\mathsf{Tor}_{\mathbb{F}[t]}(\mathcal{ I }_{m,k},\mathcal{ I }_{n,l}) &\cong \ker\Big(\mathcal{ I }_{m+k}\otimes_{\mathbb{F}[t]} \mathcal{ I }_{n,l} \rightarrow  \mathcal{ I }_{m} \otimes_{\mathbb{F}[t]} \mathcal{ I }_{n,l}\Big)\\
	&\cong \ker\Big(\mathcal{ I }_{m+k+n,l} \rightarrow  \mathcal{ I }_{m+n,l}\Big)\\
	&\cong \mathcal{ I }_{m+n+\max\{k,l\}, \min\{k,l\}}
	\end{align*}
	where the last isomorphism  is justified by the following argument: the element $t^{m+p+n} \in \mathcal{ I }_{m+k+n,l} $ goes to $t^pt^{m+n} \in \mathcal{ I }_{m+n,l}$ where $p \geq k$. The latter is zero if and only if $p \geq l$, or said equivalently  $p \geq \max \lbrace k,l\rbrace$. This finishes the proof.
\end{proof}

\begin{remark}
This last proposition recovers the formulas shown in (\ref{eq:TensorTorCalculation}). Moreover, using bilinearity,
these isomorphisms  can be used to compute the tensor product and  Tor of any pair of finitely generated graded $\F[t]$-modules.
\end{remark}

\subsection{The Graded Algebraic K\"{u}nneth Formula}
The Algebraic K\"{u}nneth Formula for  graded modules over a graded PID $R$ is proved in exactly the same way as Theorem \ref{AKT},
by noticing that all the steps  can be carried out in a degree-preserving fashion.
A few versions of this have appeared recently, see for instance  \cite[Proposition 2.9]{algkunneth} or \cite[Theorem 10.1]{pershomologicalalgebra},
 and we  include it next for completeness:

\begin{theorem}\label{algkunnethpers}
Let $C$ and $C'$ be two chain complexes of
graded modules over a graded PID $R$, and assume that one of $C,C'$ is flat.
Then, for each $n\in \N$, there is a natural short exact sequence
\begin{align*}
0
\longrightarrow
\bigoplus_{i+j = n} H_i(C) \otimes_{R} H_j(C')\;
\xrightarrow{\;\;\mu\;\;}
H_n(C \otimes_{R} C')
\xrightarrow{\;\;\nu \;\;} \hspace{3cm}\\
\hspace{2cm}
\bigoplus_{i + j = n}\mathsf{Tor}_{R}(H_i(C),H_{j-1}(C'))
\longrightarrow
0
\end{align*}
which splits, though  not naturally.
\end{theorem}

It is at this point that we depart from the existing literature on persistent K\"unneth
formulas for the tensor product of filtered chain complexes.
Indeed, next we will turn these  algebraic results into theorems at the level of filtered topological spaces.
We begin with the following result:
\begin{theorem}
	Let $\mathcal{X}, \mathcal{Y} \in \mathbf{Top^N}$.
Then, for all $n\in \N$, we have a natural short exact sequence
\begin{align*}
0
\rightarrow
\bigoplus_{i+j = n }
PH_i(\mathcal{X}) \otimes_{\mathbb{F}[t]} PH_{j}(\mathcal{Y})
\rightarrow
H_n \left(PS_{\ast}(\mathcal{ T}\left( \mathcal{X} \right)) \otimes_{\mathbb{F}[t]} PS_{\ast}(\mathcal{ T}\left(\mathcal{Y}\right))\right) \rightarrow \hspace{1cm}\\
	\bigoplus_{i+j = n} \mathsf{Tor}_{\mathbb{F}[t]}\left(PH_i(\mathcal{X}),PH_{j-1}(\mathcal{Y})\right) \rightarrow 0
	\end{align*}
which 	splits, but not naturally.
\end{theorem}
\begin{proof}
Since $\mathcal{T(X)}$ and $\mathcal{T(Y)}$ are filtered spaces,
then $ PS_*(\mathcal{T(X)})$ and $PS_*(\mathcal{T(Y)})$
are chain complexes of free graded $\F[t]$-modules, and hence flat.
Now, each inclusion $\iota_i : X_i \hookrightarrow \mathcal{T}_i(\mathcal{X})$ is a homotopy equivalence,
and thus induces an isomorphism at the level of homology.
One thing to note is that if $i < j$, then
\[
\begin{tikzcd}
\mathcal{T}_i(\mathcal{X}) \arrow[r, hook] & \mathcal{T}_j(\mathcal{X})\\
X_i \arrow[r, "\mathcal{X}(i < j)"'] \arrow[u, "\iota_i", hook]& X_j \arrow[u, "\iota_j"', hook]
\end{tikzcd}
\]
commutes only up to homotopy, but this is enough to conclude that $\iota = \{\iota_i\}_{i\in \N}$
induces  natural isomorphisms $\iota_* : H_n(\mathcal{X}; \F) \longrightarrow H_n(\mathcal{T(X)}; \F)$, $n\in \N$.
The result follows from plugging $C= PS_*(\mathcal{T(X)})$ and $C' = PS_*(\mathcal{T(Y)})$ into Theorem \ref{algkunnethpers},
and using the natural $\F[t]$-isomorphisms
$P\iota_* : PH_i(\mathcal{X};\F) \longrightarrow PH_i(\mathcal{T(X)};\F)$.
\end{proof}

\subsection{A Persistent Eilenberg-Zilber theorem}
Our next objective is to show that the chain complexes
$PS_*(\mathcal{T(X)} ) \otimes_{\F[t]} PS_*(\mathcal{T(Y)})$
and
$PS_*(\mathcal{X}\otimes_\mathbf{g} \mathcal{Y})$
have naturally
isomorphic
homology.
The result follows, as we will show in Theorem \ref{thm:PersistentEZ},
from applying the machinery of acyclic models \cite{acyclic}
to the functors
\begin{align*}
E: \mathbf{FTop}_0 \times \mathbf{FTop}_0 &\longrightarrow \textbf{gCh}_{\mathbb{F}[t]}\\
(\mathcal{X},\mathcal{Y}) &\mapsto PS_{\ast}(\mathcal{X}) \otimes_{\mathbb{F}[t]} PS_{\ast}(\mathcal{Y})\\[.2cm]
F : \mathbf{FTop}_0 \times \mathbf{FTop}_0 &\longrightarrow \textbf{gCh}_{\mathbb{F}[t]}\\
(\mathcal{X},\mathcal{Y}) &\mapsto \overline{PS}_{\ast}(\mathcal{X} \otimes \mathcal{Y})
\end{align*}
Here $\mathbf{FTop}_0$ is the full subcategory of $\mathbf{FTop}$ comprised
of those filtered spaces $\mathcal{X} = \{X_j\}_{j\in \N}$ where
$X_j$ is an open subset of $X_{j+1}$ for all $j\in \N$, and
  $\overline{PS}_{n}(\mathcal{X} \otimes \mathcal{Y})$ is the graded $\F[t]$-module whose degree-$k$ component is the sum
of vector spaces
\[
 S_n(X_k \times Y_0 ; \F)
\;+\;
S_n(X_{k-1} \times Y_1 ; \F)
\;+ \cdots +\;
S_n(X_1 \times Y_{k-1} ; \F)
\;+\;
S_n(X_0 \times Y_k ; \F),
\]
where each $S_n(X_{k- \ell} \times Y_\ell ; \F)$  is regarded as a linear subspace of $S_n\big((\mathcal{X}\otimes \mathcal{Y})_k ; \F\big)$.
The models in question are a family  $\mathcal{M} $ of objects from $\mathbf{FTop}_0\times \mathbf{FTop}_0$
satisfying certain properties (Lemmas \ref{FandEacyclic} and \ref{FandEfree}) implying
our Persistent Eilenberg-Zilber Theorem \ref{EZnew}.
We start by defining $\mathcal{M}$.

For $X\in \mathbf{Top}$,
let $\Sigma_e X \in \mathbf{FTop}_0$ be the filtered space with the empty set $\emptyset$ at indices $0 \leq i < e   $,  and $X$ for $i \geq e$:
\begin{align*}
\Sigma_e X: \ \emptyset   \subset \cdots \subset \emptyset \subset X \subset X \subset  \cdots
\end{align*}
Let $\Delta^p \subset \R^{p+1}$ be the standard geometric $p$-simplex, and let $\mathcal{M} $ be the set of objects  $(\Sigma_e \Delta^p, \Sigma_f \Delta^q)$  in $ \mathbf{FTop}_0 \times \mathbf{FTop}_0$ for $p,q ,e,f \in \N$.
We have the following:

\begin{lemma}\label{FandEacyclic}
	Each $M \in \mathcal{M}$ is both $E$-acyclic and $F$-acyclic.
That is, the homology groups $H_n(E(M))$ and $H_n(F(M))$ are trivial for all $n > 0$.
\end{lemma}
\begin{proof}
If $ M = \left(\Sigma_e \Delta^p , \Sigma_f \Delta^q\right)$, then
\[
F\left(M \right)
=
\overline{PS}_*\left(\Sigma_e \Delta^p \otimes \Sigma_f \Delta^q\right)
\cong
PS_*(\Sigma_{e+ f}\Delta^p \times \Delta^q )
\] as  chain complexes of  graded $\F[t]$-modules.
Since  $\Delta^p \times \Delta^q$ is convex,
then
\[
H_n(F(M)) \cong PH_n(\Sigma_{e +f}\Delta^p \times \Delta^q) = 0 \] for all $n >0$, showing that $M$ is $F$-acyclic.

 To show that $M$ is $E$-acyclic,
notice that $PH_n\left(\Sigma_e \Delta^p\right) = 0$ for
$n > 0$, and that $PH_0\left(\Sigma_e \Delta^p \right) \cong t^e \F[t]$ is a free graded $\F[t]$-module.
By the graded algebraic K\"unneth Theorem \ref{algkunnethpers},
with $C = PS_*(\Sigma_e \Delta^p)$ and $C' = PS_*(\Sigma_f \Delta^q)$,
we get that
\[
H_n(E(M)) = H_n( PS_*(\Sigma_e \Delta^p) \otimes_{\F[t]} PS_*(\Sigma_f \Delta^q) ) = 0
\]
for all $n>0$, since the tensor products and Tor modules in the short exact sequence are all zero.
\end{proof}	
\begin{lemma}\label{FandEfree}
	For each $n\in \N$, the functors $F_n$ and $E_n$ are free with base in $\mathcal{ M }$.
\end{lemma}
\begin{proof}
Let $n\in \N$. The functor $F_n$ (resp. $E_n$) being free with base in $\mathcal{M}$ means that
    for all $(\mathcal{X,Y})\in \mathbf{FTop}_0\times \mathbf{FTop}_0$, the graded $\F[t]$-module
    $F_n(\mathcal{X,Y})  $ (resp. $E_n(\mathcal{X,Y})$) is freely generated over
    $\F[t]$ by some subset of
\[
\bigcup_{\substack{M \in \mathcal{M} \\ u: M \rightarrow (\mathcal{X,Y})} }
\mathsf{Img}
\Big( F_n(u): F_n(M) \longrightarrow F_n(\mathcal{X,Y})\Big)
\]
    where $u$ ranges over all morphisms in $\mathbf{FTop}_0 \times \mathbf{FTop}_0$ from $M\in \mathcal{M}$ to $(\mathcal{X,Y})$.

Starting with $F$, we will first construct a free  $\F[t]$-basis for
\[
F_n(\mathcal{X,Y}) = \overline{PS}_n(\mathcal{X\otimes Y}) = \bigoplus_{k\in \N} \sum_{i + j = k} S_n(X_i \times Y_j)
\]
consisting of singular $n$-simplices on the spaces $X_i\times Y_j$.
We proceed inductively as follows:
Let $\mathcal{S}_0$ be the collection of all singular $n$-simplices
$ \Delta^n \rightarrow X_0 \times Y_0$, and assume
that for a fixed $k\geq 1$ and every $0 \leq \ell < k $ we have constructed
a set $\mathcal{S}_\ell$ of simplices $\Delta^n \rightarrow X_p\times Y_q$, $p + q = \ell$,
so that
\[
\mathcal{A}_{k-1} \;=\;
t^{k-1}\mathcal{S}_0
\;\cup\;
t^{k-2}\mathcal{S}_1
\;\cup\;
\cdots
\;\cup\;
t \mathcal{S}_{k-2}
\;\cup\;
\mathcal{S}_{k-1}
\]
is a basis, over $\F$, for
\[
\overline{PS}_n(\mathcal{X\otimes Y})_{k-1} = \sum\limits_{p+q = k-1} S_n(X_p \times X_q).
\]

Since $\overline{PS}_n(\mathcal{X\otimes Y})$
is an $\F[t]$-submodule of  $PS_n(\mathcal{X\otimes Y})$ and the latter is torsion free,
then   $t\mathcal{A}_{k-1}$ is an $\F$-linearly independent subset of
$\overline{PS}_n(\mathcal{X\otimes Y})_{k}$.
We claim that $t\mathcal{A}_{k-1}$ can be completed to
an $\F$-basis $\mathcal{A}_k$ of $\overline{PS}_n(\mathcal{X\otimes Y})_{k}$,
consisting of singular $n$-simplices $\Delta^n \rightarrow X_i \times Y_j$, $i +j =k$.
Indeed, if $\mathscr{C}$ is the collection of all   sets $C$ of
simplices $\Delta^n \rightarrow X_i\times Y_j$, $i + j = k$, so that
$C$ is $\F$-linearly independent in $\overline{PS}_n(\mathcal{X\otimes Y})_{k}$ and
 $t \mathcal{A}_{k-1} \subset C$, then
$\mathscr{C}$ can be ordered by inclusion, and any chain in $\mathscr{C}$ is bounded  above by its
union. By Zorn's lemma $\mathscr{C}$ has a maximal element  $\mathcal{A}_k$,   which  yields the desired basis.
If we let $\mathcal{S}_k  = \mathcal{A}_k \smallsetminus t\mathcal{A}_{k-1}$, then
\[
\mathcal{S} := \bigcup_{k\in \N} \mathcal{S}_k
\]
is a free $\F[t]$-basis for $\overline{PS}_n(\mathcal{X\otimes Y})$.

We claim that each
 $\sigma \in \mathcal{S}$ is in the image of $F_n(u)$ for some model $M\in \mathcal{M}$ and some morphism $u : M \rightarrow (\mathcal{X,Y})$
in $\mathbf{FTop}_0^2$.
Indeed, since $\sigma$ is of the form $\sigma : \Delta^n \rightarrow X_i \times Y_j$ for some $i,j \in \N$,
then there are unique
singular simplices $\sigma^X_i : \Delta^n \rightarrow X_i$ and $\sigma^Y_j : \Delta^n \rightarrow Y_j$,
so that if
 $d: \Delta^n \rightarrow \Delta^n \times \Delta^n$ is the diagonal map $d(a) = (a,a)$,
then $\sigma = \left(\sigma_i^X , \sigma_j^Y\right)\circ d$.
Let $\Sigma \sigma_i^X : \Sigma_i \Delta^n \rightarrow \mathcal{X}$ be the morphism in $\mathbf{FTop}_0$
defined as the inclusion $\emptyset \hookrightarrow X_\ell$ for $\ell < i$, and as
the composition   $\Delta^n \xrightarrow{\sigma_i^X} X_i \hookrightarrow X_\ell$ for $\ell \geq i$.
Define  $\Sigma \sigma_j^Y: \Sigma_j \Delta^n \rightarrow \mathcal{Y}$ in a similar fashion.
Then,
$
\left(\Sigma \sigma_i^X, \Sigma \sigma_j^Y\right):
\left(\Sigma_i \Delta^n , \Sigma_j \Delta^n\right)
\longrightarrow
(\mathcal{X,Y})$ is a morphism in $\mathbf{FTop}_0^2$,
and  if
\[
d_{i,j} \in S_n\big((\Sigma_i \Delta^n)_i \times (\Sigma_j\Delta^n)_j\big)
= S_n(\Delta^n \times \Delta^n)
\]
is the singular $n$-simplex corresponding to the diagonal map $d$,
then
\[
F_n\left(\Sigma \sigma_i^X , \Sigma \sigma_j^Y\right) (d_{i,j}) = \sigma.
\]

As for the freeness of
\[
E_n(\mathcal{X,Y}) = \bigoplus_{p + q =n} PS_p(\mathcal{X})\otimes_{\F[t]} PS_q(\mathcal{Y}), 
\]
we note that each direct summand $PS_p(\mathcal{X})\otimes_{\F[t]} PS_q(\mathcal{Y})$ is freely generated over $\F[t]$
by elements which can be written as the tensor product of a homogeneous element from $PS_p(\mathcal{X})$ and a homogeneous element
from $PS_q(\mathcal{Y})$.
That is, by elements of the form $\sigma_p^X \otimes_{\F[t]} \sigma_q^Y$,
where $\sigma_p^X : \Delta^p \rightarrow X_e$ and
 $\sigma_q^Y : \Delta^q \rightarrow Y_f$ are singular simplices.
We let $\Sigma_e\sigma_p^X : \Sigma_e\Delta^p \rightarrow \mathcal{X}$ and
 $\Sigma_f\sigma_q^Y : \Sigma_f\Delta^q \longrightarrow \mathcal{Y}$ be defined
as in the analysis of $F_n$ above.
Let $i_p \in PS_p(\Sigma_e \Delta^p)$ be the homogeneous element
of degree $e$ corresponding to the identity of $\Delta^p$, and define $i_q \in PS_q(\Sigma_f \Delta^q)$  in a similar fashion.
Then,
$i_p \otimes_{\F[t]} i_q \in E_n(\Sigma_e \Delta^p , \Sigma_f \Delta^q)$ and
\[
\sigma_p^X \otimes_{\F[t]} \sigma_q^Y \; =\;
 E_n\left(\Sigma_e \sigma^X_p , \Sigma_f \sigma^Y_q\right) (i_p \otimes_{\F[t]} i_q),
\]
thus completing the proof.
\end{proof}	
\begin{lemma}\label{HFEequivalence}
	The functors  $H_0F $ and $  H_0E$ are naturally equivalent.
\end{lemma}
\begin{proof}
Given $(\mathcal{ X },\mathcal{ Y }) \in \mathbf{FTop}_0 \times \mathbf{FTop}_0$,
our goal is to define a natural isomorphism
\[
\Psi:
H_0\left(PS_*(\mathcal{X}) \otimes_{\F[t]} PS_*(\mathcal{Y})\right)
\longrightarrow
H_0\left(\overline{PS}_*(\mathcal{X\otimes Y})\right).
\]
The first thing to note is that  the graded algebraic K\"unneth formula (Theorem \ref{algkunnethpers})
yields a natural isomorphism
\[
H_0\left(PS_*(\mathcal{X}) \otimes_{\F[t]} PS_*(\mathcal{Y})\right)
\cong
PH_0(\mathcal{X}) \otimes_{\F[t]} PH_0(\mathcal{Y}).
\]
Moreover, since each $X_{k - \ell}\times Y_\ell$  is open in $(\mathcal{X\otimes Y})_k$,   then the inclusion
\[
\sum_{i+j = k} S_n(X_i \times Y_j) \hookrightarrow S_n((\mathcal{X\otimes Y})_k)
\]
is a chain homotopy equivalence (see  \cite[Proposition 2.21]{hatcher}),
and thus
\[
H_n\left(\overline{PS}_*(\mathcal{X\otimes Y})\right)
\cong PH_n(\mathcal{X\otimes Y}).
\]
Therefore, it is enough to construct a natural isomorphism
\[
\Psi: PH_0(\mathcal{X})\otimes_{\F[t]}PH_0(\mathcal{Y})
\longrightarrow
PH_0(\mathcal{X\otimes Y})
\]
and we will do so on each degree $k$ as follows.
Recall that (see Proposition \ref{gradedtensor})
\[
\left( PH_0(\mathcal{X})\otimes_{\F[t]}PH_0(\mathcal{Y}) \right)_k
\cong
\left( \bigoplus_{i+j = k} H_0(X_i) \otimes_{\F} H_0(Y_j) \right)\Big/J_k
\]
where $J_k$ is the linear subspace  of
$\bigoplus\limits_{i+j = k} H_0(X_i) \otimes_{\F} H_0(Y_j)$
generated by elements of the form\;
$(t^\ell a) \otimes_{\F} b - a \otimes_{\F} (t^\ell b)$\;
 with \;  $\mathsf{deg}(a) + \mathsf{deg}(b) + \ell = k$.
Let
\[
\begin{array}{rrcl}
 \psi_k : & \bigoplus\limits_{i+j =k } H_0(X_i) \otimes_{\F} H_0(Y_j) & \longrightarrow & H_0((\mathcal{X\otimes Y})_k) \\
&\Gamma_i^X \otimes_\F \Gamma_j^Y& \mapsto & \Gamma
\end{array}
\]
where $\Gamma_i^X$  (resp. $\Gamma_j^Y$) is a path-connected component of $X_i$ (resp. $Y_j$),  $i +j = k$,
and $\Gamma $ is the unique path-connected component of $(\mathcal{X\otimes Y})_k$ containing $\Gamma_i^X \times \Gamma_j^Y$.
This definition on basis elements uniquely determines $\psi_k$ as an $\F$-linear map, and we claim
that it is surjective with $\mathsf{ker}(\psi_k) = J_k$.
Surjectivity follows from observing that if $(x,y) \in \Gamma$, then
$(x,y) \in X_i \times Y_j$ for some $i+j = k$, and that if
$\Gamma_i^X$ and $\Gamma_j^Y$ are the path-connected components in $X_i$ and $Y_j$, respectively,
so that $(x,y) \in \Gamma_i^X \times \Gamma_j^Y$,
then  path-connectedness of $\Gamma_i^X \times \Gamma_j^Y $ and maximality of $\Gamma$ imply
$\Gamma_i^X \times \Gamma_j^Y \subset \Gamma$.

The inclusion $J_k \subset \mathsf{ker}(\psi_k)$ is immediate, since it holds true
for elements of the form
$\left(t^\ell \Gamma_{k - j}^X \otimes_\F \Gamma_{j-\ell}^Y - \Gamma_{k - j}^X \otimes_\F t^\ell \Gamma_{j-\ell}^Y\right) $,
and these generate $J_k$ over $\F$.
In order to show that $\mathsf{ker}(\psi_k) \subset  J_k$, we will
first establish the following: \\

\noindent \textbf{\underline{Claim 1}:}
 If $\Gamma_i^X \times \Gamma_j^Y \subset X_i\times Y_j$
and $\Gamma_p^X \times \Gamma_q^Y \subset X_p\times Y_q$  are path-connected components, $p + q = i + j = k$,
such that $\psi_k\left(\Gamma_i^X \otimes_\F \Gamma_j^Y\right) = \psi_k\left(\Gamma_p^X \otimes_\F \Gamma_q^Y\right)$,
then
\[
\left(
\Gamma_i^X \otimes_\F \Gamma_j^Y - \Gamma_p^X \otimes_\F \Gamma_q^Y \right)
\in J_k.
\]
Indeed, let
$
\gamma : [0,1] \longrightarrow (\mathcal{ X } \otimes \mathcal{ Y })_k
$ be a path
with
$\gamma(0)  \in \Gamma^X _i \times \Gamma^Y_j$
and
$\gamma(1)   \in \Gamma^X _p \times \Gamma^Y_q$.
Since each $X_{k -\ell} \times Y_{\ell}$ is open in
$(\mathcal{ X } \otimes \mathcal{ Y })_k$,
then   the Lebesgue's number lemma
yields  a partition
 $0 = r_0 < r_1 < \dots < r_N = 1$
such that
$\gamma\left([r_n, r_{n+1}]\right) \subset X_{i_n} \times Y_{j_n} $
and $i_n + j_n = k$ for all $0 \leq n < N$.
Let
$(x_n,y_n) =
 \gamma\left(r_n\right)$,
and let
$  \Gamma^X _{i_{n}} \times \Gamma^Y_{j_{n}}$
be its path-connected component in $X_{i_n}\times Y_{j_n}$.
Since $(i,j)= (i_0 , j_0)$ and
$(p,q) = (i_N , j_N)$, then it is
enough to show that
\[
\left(
\Gamma_{i_n}^X \otimes_\F \Gamma_{j_n}^Y -
\Gamma_{i_{n+1}}^X \otimes_\F \Gamma_{j_{n+1}}^Y
\right)
\in J_k
\;\;\;\;\mbox{ for all }\;\;\;\;
n=0, \ldots, N-1.
\]
Fix $0 \leq n < N$, assume that $i_n \leq i_{n+1}$ , and let $\ell = i_{n+1} - i_n = j_n - j_{n+1}$ (the argument is similar if $i_n \geq i_{n+1}$).
Since the projection
$\pi^\mathcal{ X }\left(\gamma\left([r_n, r_{n+1}]\right)\right) \subset X_{i_n}$
is a path in $X_{i_n}$  from $x_n$ to $x_{n+1}$, then  $t^{\ell}\Gamma^X _{i_n} = \Gamma^X _{i_{n+1}}$
in $H_0(X_{i_{n+1}})$.
Similarly, the projection
$\pi^\mathcal{ Y }\left(\gamma\left([r_n, r_{n+1}]\right)\right) \subset Y_{j_n}$
is a path in $Y_{j_n}$ from $y_n$ to $y_{n+1}$, and since $j_{n+1} \leq j_n$, then
$t^{\ell}\Gamma^Y_{j_{n+1}} = \Gamma^Y_{j_{n}}$ in $H_0(Y_{j_n})$.
This calculation shows that
\[
\left(
\Gamma^X _{i_n} \otimes_{\F} \Gamma^Y_{j_n}
-
 \Gamma^X _{i_{n+1}} \otimes_{\F} \Gamma^Y_{j_{n+1}}
\right)
=
\left(
 \Gamma^X _{i_n} \otimes_{\F} t^{\ell}\Gamma^Y_{j_{n+1}}
-
  t^{\ell} \Gamma^X _{i_n} \otimes_{\F} \Gamma^Y_{j_{n+1}}
\right)
\in J_k
\]
and therefore $\left(
\Gamma^X _{i} \otimes_{\F} \Gamma^Y_{j}
-
 \Gamma^X _{p} \otimes_{\F} \Gamma^Y_{q}
\right)\in J_k$. (\emph{end of Claim 1})

Let us now show that $\mathsf{ker}(\psi_k) \subset J_k$.
To this end, let   $c\in \mathsf{ker}(\psi_k)$ and write
\[
c = \sum_{n =1 }^N c_n \cdot \left(\Gamma_{i_n}^X \otimes_{\F} \Gamma_{j_n}^Y\right)
\]
where $c_n \in \F$, and  where the indices $(i_n , j_n)$ have been ordered such that there are
integers $0 = n_0 < n_1 < \cdots < n_D = N $, and distinct
path-connected components $\Gamma_1, \ldots, \Gamma_D$ of $(\mathcal{X\otimes Y})_k$
so that if $n_{d-1} < n \leq n_d$, then
$\psi_k\left(\Gamma_{i_n}^X \otimes_\F \Gamma_{j_n}^Y\right) = \Gamma_d$  for all $d = 1 , \ldots, D$.
Since the $\Gamma_d$'s are linearly independent in $H_0((\mathcal{X\otimes Y})_k)$, then
$\psi_k(c) = 0$ implies
\[
c_{n_d} =
\sum_{n = 1 + n_{d-1}}^{n_d - 1} -c_n  \;\;\;\; , \;\;\; d = 1,\ldots, D
\]
and therefore
\[
c = \sum_{d = 1}^D \;\sum_{n = 1+ n_{d-1}}^{n_d - 1} c_n\cdot
\left(
\Gamma^X_{i_n}\otimes_\F \Gamma^Y_{j_n} - \Gamma_{i_{n_d}}^X\otimes_\F \Gamma_{j_{n_d}}^Y
\right)
\]
where each summand is an element of $J_k$, by Claim 1 above.
Thus, $\mathsf{ker}(\psi_k) = J_k$, $\psi_k$ induces a natural $\F$-isomorphism
\[
\Psi_k : \left( PH_0(\mathcal{X})\otimes_{\F[t]}PH_0(\mathcal{Y}) \right)_k
\longrightarrow H_0((\mathcal{X\otimes Y})_k)\]
and letting $\Psi = \Psi_0 \oplus \Psi_1 \oplus \cdots$ completes the proof.
\end{proof}	
We now move onto the main results of this section.

\begin{lemma}[Persistent Eilenberg-Zilber]\label{EZnew}
	For objects $\mathcal{X}, \mathcal{Y} \in \mathbf{FTop}_0$ and coefficients in a field $\mathbb{F}$, there is a natural graded chain equivalence
\[
\zeta:
PS_*(\mathcal{X}) \otimes_{\F[t]} PS_*(\mathcal{Y})
\longrightarrow
\overline{PS}_*(\mathcal{X\otimes Y})
\]
	with $H_0(\zeta) = \Psi$ (Lemma \ref{HFEequivalence}), and  unique up to a graded chain homotopy.
\end{lemma}
\begin{proof}
The proof follows exactly that of Theorem \ref{thm:ez} \cite{EZ}, using acyclic models \cite{acyclic}
(see also the proof of Theorem 5.3 in \cite{vick2012homology} for a more direct comparison).
Indeed, the functors in question are acyclic (Lemma \ref{FandEacyclic}),  free  (Lemma \ref{FandEfree}), and naturally equivalent at the level of zero-th homology
(Lemma \ref{HFEequivalence}).
\end{proof}

\begin{theorem}\label{thm:PersistentEZ}
 If $\mathcal{X,Y} \in \mathbf{Top^N}$, then there is a natural $\F[t]$-isomorphism
\[
H_n\left(
PS_*(\mathcal{T(X)}) \otimes_{\F[t]} PS_*(\mathcal{T(Y)})
\right)
\;\cong\;
PH_n(\mathcal{X} \otimes_\mathbf{g} \mathcal{Y}).
\]
\end{theorem}
\begin{proof}
 Lemma \ref{lemma:tensorandgentensor} implies that
$PH_n(\mathcal{X}\otimes_\mathbf{g} \mathcal{Y}) \cong PH_n(\mathcal{T(X)}\otimes \mathcal{T(Y)})$,
 and Proposition \ref{prop:defretract} shows that if $0< \epsilon < 1$, then
\[
PH_n(\mathcal{T(X)}\otimes \mathcal{T(Y)})
\cong PH_n(\mathcal{T}^\epsilon\mathcal{(X)}\otimes \mathcal{T}^\epsilon\mathcal{(Y)}).
\]
Moreover,  since $\mathcal{T}^\epsilon(\mathcal{X})\in \mathbf{FTop}_0$,
then the inclusion
\[
\overline{PS}_*(\mathcal{T}^\epsilon(\mathcal{X}) \otimes \mathcal{T}^\epsilon(\mathcal{Y}) )
\hookrightarrow
PS_*(\mathcal{T}^\epsilon(\mathcal{X}) \otimes \mathcal{T}^\epsilon(\mathcal{Y}) )
\]
is a graded chain homotopy equivalence (see  \cite[Proposition 2.21]{hatcher}),
and
 thus  Lemma \ref{EZnew} shows that
\[
PH_n(\mathcal{X}\otimes_\mathbf{g} \mathcal{Y})
\cong
H_n\left(
PS_*(\mathcal{T}^\epsilon\mathcal{(X)}) \otimes_{\F[t]} PS_*(\mathcal{T}^\epsilon\mathcal{(Y)})
\right).
\]
The $\epsilon$ can be removed since $\mathcal{T}_k^\epsilon(\mathcal{X})$ deformation retracts onto $\mathcal{T}_k\mathcal{(X)}$
(Proposition \ref{prop:defretract}), which completes the proof.
\end{proof}

Finally, we obtain

\begin{theorem}\label{KunnethTPF}
	There is a natural short exact sequence of graded $\F[t]$-modules
	\begin{align*}
	0 \rightarrow \bigoplus_{i+j  = n}\left(PH_i(\mathcal{X}; \mathbb{F}) \otimes_{\mathbb{F}[t]} PH_{j}(\mathcal{Y}; \mathbb{F})\right) \rightarrow PH_n \left( \mathcal{X} \otimes_{\emph{\textbf{g}}} \mathcal{Y} \right) \rightarrow\\
	\bigoplus_{i+j = n}\mathsf{Tor}_{\mathbb{F}[t]}\left(PH_i(\mathcal{X}; \mathbb{F}),PH_{j-1}(\mathcal{Y}; \mathbb{F})\right) \rightarrow 0
	\end{align*}
which splits, though not naturally.
\end{theorem}
\begin{remark}
	If $\mathbf{S}$ is either $\mathbf{Top}$, $\mathbf{Met}$, $\mathbf{oSimp} $, or $\mathbf{Simp}$, then following Remark \ref{rmk:spaceskunneth} yields a similar Persistent
K\"unneth theorem in each case.
\end{remark}

\subsection{Barcode Formulae}
We now give the barcode formula for the generalized tensor product of objects in $\mathbf{S^N}$.
Recall from Theorem \ref{thm:CrawleyBoevey} that pointwise finite objects in $\mathbf{Mod_\mathbb{F}^N}$ can be uniquely decomposed into interval diagrams.
In Proposition \ref{intervaltensor}, the tensor product and $\mathsf{Tor}$ of the associated interval modules was computed, and
the following is a consequence of Theorem \ref{KunnethTPF}.
\begin{theorem} Let $\mathcal{X},\mathcal{Y} \in \mathbf{S^N}$. Assume that for $ 0 \leq i,j \leq n$, the diagrams $H_i(\mathcal{ X })$ and $H_j(\mathcal{ Y })$ are pointwise finite. Then  $H_n(\mathcal{X} \otimes_\mathbf{g} \mathcal{Y})$ is pointwise finite with barcode
	\begin{align*}
	\mathsf{bcd}_n (\mathcal{X} \otimes_{\emph{\textbf{g}}} \mathcal{Y}) =  \bigcup_{i + j= n} 	\bigg\{  (\ell_J + I) \cap (\ell_I + J)\ \Big\vert \ I \in \mathsf{bcd}_{i}(\mathcal{X}), J \in \mathsf{bcd}_{j}(\mathcal{Y}) \bigg\} \hspace{1cm}\\
	\bigcup \hspace{5.2cm}\\
	\bigcup_{i+j = n}\bigg\{ (\rho_J + I)\cap(\rho_I + J)\Big\vert \ I \in \mathsf{bcd}_{i}(\mathcal{X}), J \in \mathsf{bcd}_{j-1}(\mathcal{Y}) \bigg\}.\hspace{0.9cm}
	\end{align*}
\end{theorem}
In the above formula, the first union of intervals is associated with the tensor part of the sequence in theorem \ref{KunnethTPF}, while the second union is associated with the $\mathsf{Tor}$ part. The results are easy to generalize to $0$-th persistent homology in the case of finite tensor product of objects in \textbf{Top}.
\begin{corollary}[$0$-th persistent homology]
	Let $\mathcal{ X }_1,   \dots, \mathcal{ X }_p$$ \in \mathbf{S^N}$, and assume that  $H_0\left(\mathcal{ X }_i\right)$ is pointwise finite for all $i\in \N$.
Let  $\mathsf{bcd}_0(\mathcal{ X }_i)$ be the $0$-dimensional barcode of $\mathcal{ X }_i$, then
	\begin{align*}
	\mathsf{bcd}_0\left(\mathcal{ X }_1 \otimes_{\emph{\textbf{g}}} \cdots \otimes_{\emph{\textbf{g}}} \mathcal{ X }_p \right)= \bigg\lbrace \bigcap_{i=1}^p \big( \left( \ell - \ell_{I_i} \right)  + I_i  \big)  \bigg\vert \ \ell = \sum_{i=1}^p \ell_{I_i}, \ I_i \in\mathsf{bcd}_0(\mathcal{ X }_i)\bigg\rbrace.
	\end{align*}
	In other words, a bar in $\mathcal{ X }_1 \otimes_{\emph{\textbf{g}}} \cdots \otimes_{\emph{\textbf{g}}} \mathcal{ X }_p$ corresponding to $p$ bars, one from each $\mathsf{bcd}_0(\mathcal{ X }_i)$, starts at the sum of the starting points and lives as long as the shortest one.
\end{corollary}

Topological inference deals with estimating the homology of  a metric space $\mathbb{X}$ from a finite point cloud $X$, as discussed in Example \ref{expl:MotivationTDA}. Usually, longer bars in a barcode represent significant topological features. The following corollary allows us to count the number of bars longer than a threshold $\epsilon > 0$ in the barcode of the tensor product.
\begin{corollary}[On high persistence]
	For $n \geq 0$ and $\epsilon>0$, let $\emph{\textbf{c}}^{\mathcal{X}}_n (\beta )$ and $\emph{\textbf{c}}^{\mathcal{X}}_n (\beta \geq \epsilon)$ denote the number of bars with length exactly $\beta$ and at least $\epsilon$ respectively in $\mathsf{bcd}_n(\mathcal{X})$. Assume that the lengths of longest bars in $\mathsf{bcd}_n(\mathcal{X})$ and $\mathsf{bcd}_n(\mathcal{Y})$ are $\beta^{\mathcal{X}}_n$ and $\beta^{\mathcal{Y}}_n$ respectively, then
	\begin{align*}
	\beta^{\mathcal{X} \otimes_{\emph{\textbf{g}}} \mathcal{Y}}_n = \max \bigg\{ \max_{k+l=n}\Big\{ \min\lbrace \beta^{\mathcal{X}}_k, \beta^{\mathcal{Y}}_l\rbrace  \Big\} , \max_{k+l=n}\Big\{ \min\lbrace \beta^{\mathcal{X}}_k, \beta^{\mathcal{Y}}_{l-1}\rbrace  \Big\} \bigg\}
	\end{align*}
	and
	\begin{align}\label{eqn:longbars}
	\emph{\textbf{c}}^{\mathcal{X} \otimes_{\emph{\textbf{g}}} \mathcal{Y}}_n(\beta \geq \epsilon) = \sum_{k+l=n} \emph{\textbf{c}}^{ \mathcal{Y}}_l(\beta \geq \epsilon)  \emph{\textbf{c}}^{\mathcal{X}}_k(\beta \geq \epsilon)
	+  \sum_{k+l=n-1} \emph{\textbf{c}}^{ \mathcal{Y}}_l(\beta \geq \epsilon)  \emph{\textbf{c}}^{\mathcal{X}}_k(\beta \geq \epsilon).
	\end{align}
\end{corollary}
By replacing $\epsilon$ with $\beta^{\mathcal{X} \otimes_{{\textbf{g}}} \mathcal{Y}}_n$, the formula in equation \ref{eqn:longbars} can be used to compute the number of bars with the longest length in $\mathsf{bcd}_n\left(\mathcal{X} \otimes_{{\textbf{g}}} \mathcal{Y}\right)$.

\section{Applications and Discussion}\label{application}
\subsection{Application: Time series analysis}
Time series are ubiquitous in data science and make for an important object of study.
Recently, the problem of detecting  recurrence in time varying-data
has received increasing attention in the applied topology literature \cite{perea2019notices, Xu2019}.
Two types of recurrence that   appear prominently are periodicity  and \emph{quasiperiodicity}.
A function $f: \R \longrightarrow \C$ is said to be quasiperiodic on the $\Q$-linearly independent frequencies
$\omega_1 ,\ldots, \omega_N$ if there exists  $F: \R^N \longrightarrow \C$
so that each  $t\mapsto F(t_1,\ldots, t_{n-1}, t, t_{n+1}, \ldots, t_N)$
is periodic with frequency $\omega_n$,  and $f(t)= F(t,\ldots, t)$ for all $t\in \R$.
Quasiperiodicity appears  naturally   in biphonation phenomena in mammals \cite{mammalsubharmonics},
and in  transitions to chaos in rotating fluids \cite{rotatingfluid}.

The topological approach to   recurrence detection   starts with
the \emph{sliding window embedding} of  $f$, defined for each $t\in \R$  and
 parameters $d \in \mathbb{N}$ and $\tau \in (0,\infty)$, as
\begin{align*}
SW_{d,\tau}f(t) =
\begin{bmatrix}
f(t) \\
f(t+\tau) \\
\vdots \\
f(t+d\tau)
\end{bmatrix} \in \C^{d+1}.
\end{align*}
One then uses the barcodes $\bcd_i^\mathcal{R}(SW_{d,\tau} f(T) ; \F)$ for $T\subset \R$,
as signatures;
for if
$f$ is periodic,  then $SW_{d,\tau} f(\R)$ is a closed curve \cite{sw1pers},
and when  $f$ is quasiperiodic  then $SW_{d,\tau} f(\Z)$ is dense in a high-dimensional torus \cite{toroidal, SWquasiperiodic}.
These ideas have been used successfully to find new periodic genes in biological systems \cite{perea2015sw1pers},
and to automatically detect biphonation in high-speed videos of vibrating vocal folds \cite{tralie2018quasi}.

One of the main challenges in computing
$\bcd_i^\mathcal{R}(SW_{d,\tau}f(T); \F)$ for a quasiperiodic time series $f$
is that if
  $SW_{d,\tau }f$ fills out the torus at a slow rate ---
requiring  a potentially large  $T \subset \Z$ ---
then
the persistent homology computation becomes prohibitively large.
We will see next how the categorical persistent K\"unneth formula (Corollary \ref{cor:RipsKunneth})
can be used to address this.
The general strategy will be illustrated with a computational example.

Let   $c_1,c_2  \in \C\smallsetminus \{0\}$ with  $|c_1|^2 +  |c_2|^2 = 1$,
and let $\omega \in \R\smallsetminus \Q$.
If
\[
f(t) =  c_1 e^{i t} +  c_2 e^{i \omega t }
\]
then
$
SW_{d,\tau} f(t) =
\Omega \cdot \phi(t)
$,
where
\[
\Omega =
\frac{1}{\sqrt{d+1}}
\begin{bmatrix}
  1 & 1  \\
  e^{i\tau} & e^{i \omega\tau}\\
  \vdots & \vdots \\
 e^{i d\tau} & e^{i d\omega \tau}
\end{bmatrix}
\;\;\; , \;\;\;\;\;
\phi(t) =
\sqrt{d+1}
\begin{bmatrix}
  c_1e^{i t} \\[.2cm]
  c_2e^{i \omega t }
\end{bmatrix}.
\]
An elementary calculation shows that if $d\geq 1$ and $\tau = \frac{2\pi}{(d+1)|\omega - 1|}$, then the columns of $\Omega$ are orthonormal in $\C^{d+1}$.
Fix $d = 1$, $\omega =\sqrt{3}$, $\tau$ as above, and $c_1 = c_2 = 1/\sqrt{2}$.
Figure \ref{fig:timeseriesandPCA} below shows the real part of $f$ (left),
and a visualization (right) of the \emph{sliding window point cloud}  $\mathbb{SW}_{d,\tau}f = SW_{d,\tau} f(T)$,
$T = \{ t\in \Z : 0 \leq t \leq 4,000\}$,
 using Principal Component Analysis (PCA) \cite{jolliffe2011principal}.
\begin{figure}[!htb]
	\centering
    \hspace{-.2cm}
	\begin{subfigure}{0.65\textwidth}
		\centering
		\includegraphics[width = \textwidth]{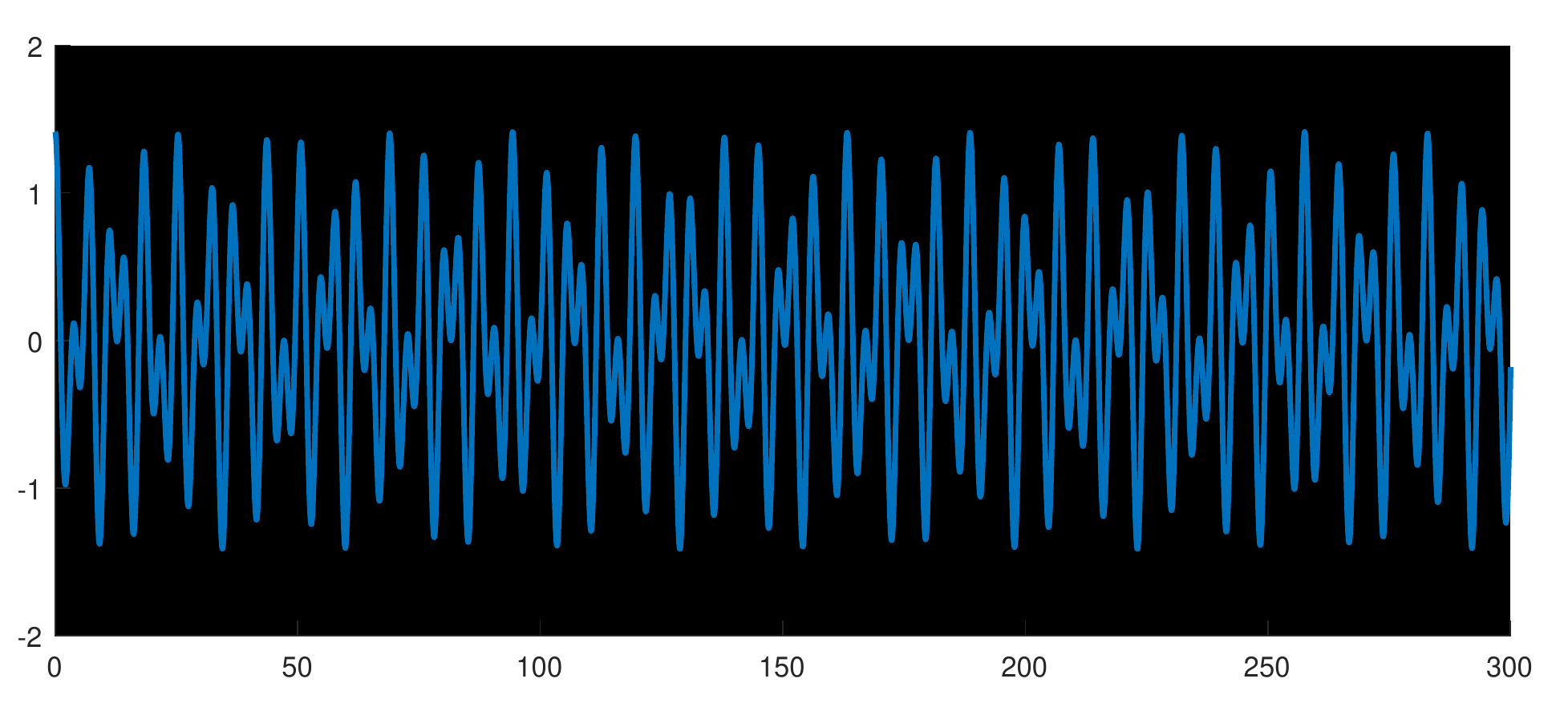}
		\end{subfigure}
    \hspace{.2cm}
	\begin{subfigure}{0.32\textwidth}
		\centering
		\includegraphics[width = \textwidth]{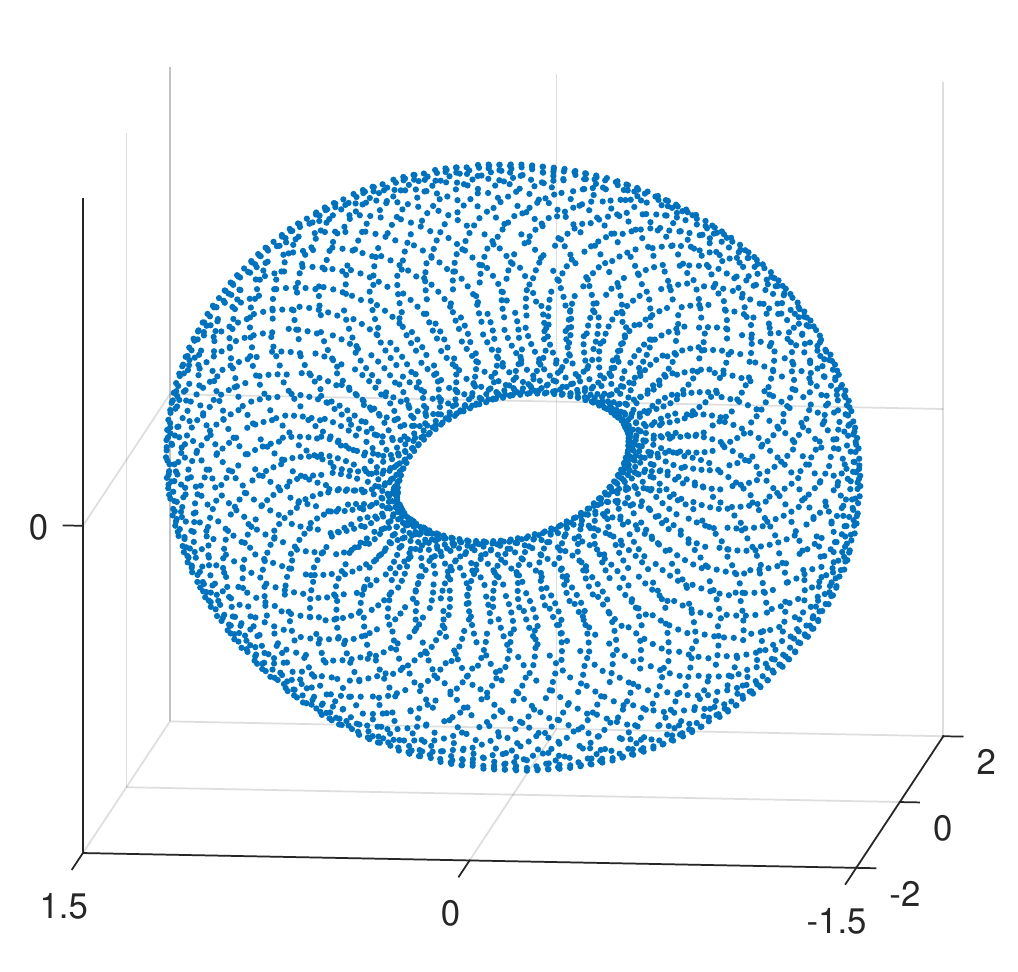}
	\end{subfigure}
	\caption{ (Left)   $\mathsf{real}\big(f(t)\big) = \frac{1}{\sqrt{2}}\cos(t) + \frac{1}{\sqrt{2}}\cos(\sqrt{3}t)$, and
(Right) a PCA projection into $\R^3$ of the point cloud $\mathbb{SW}_{d,\tau}f$.}
	\label{fig:timeseriesandPCA}
\end{figure}

Computing the exact barcodes $\bcd_i^\mathcal{R}(\mathbb{SW}_{d,\tau} f )$, $i = 0 ,1,2$,
for point clouds of this size is already prohibitively large  (e.g., with \verb"Ripser" \cite{ripser}).
Thus, one strategy is to choose a smaller set of landmarks
$L\subset\mathbb{SW}_{d,\tau} f$ and use the stability theorem \cite{chazal2016structure}
\[
d_B\left(\bcd_i^\mathcal{R}( L) ,  \bcd_i^\mathcal{R}(\mathbb{SW}_{d,\tau} f) \right)
\leq
2d_{GH}(L , \mathbb{SW}_{d,\tau} f )
\]
to estimate the barcodes of $\mathbb{SW}_{d,\tau} f$ using those of $L$, up to
an error (in the bottleneck sense) of twice the Gromov-Hausdorff distance $d_{GH}$ between
$L$ and $\mathbb{SW}_{d,\tau} f$.
Specifically, if $r = d_{GH}(L , \mathbb{SW}_{d,\tau} f)$ and $I = [\ell_I , \rho_I) \in \bcd_i^\mathcal{R}(L)$
satisfies $\rho_I - \ell_I > 4r$, then there is a unique $J \in \bcd_{i}^\mathcal{R}(\mathbb{SW}_{d,\tau} f)$
with $|\rho_I - \rho_J|, |\ell_I  - \ell_J| \leq 2r$, and thus
\begin{align}\label{eqn:landmarkbounds}
\begin{split}
\max\{0\;,\; \rho_I - 2r\} & \leq \rho_J \leq \rho_I + 2r \\[.1cm]
\max\{0\;,\; \ell_I - 2r\} & \leq \ell_J \leq \ell_I + 2r.
\end{split}
\end{align}
In what follows we compare this landmark approximation procedure
to a different strategy  leveraging the categorical persistent K\"unneth theorem
for the Rips filtration on the maximum metric (Corollary \ref{cor:RipsKunneth}).
Here is the setup:
\begin{description}
    \item[Landmarks] We select a   set $L \subset \mathbb{SW}_{d,\tau} f$ of 400 landmarks ($\%10$ of the data)
     using \verb"maxmin" sampling. That is,   $\ell_1 \in \mathbb{SW}_{d,\tau} f$ is chosen at random
     and   the rest of the landmarks are selected inductively as
    \[
    \ell_{j +1}  = \argmax_{x \in \mathbb{SW}_{d,\tau} f } \min\big\{\| x - \ell_1\|_2 , \ldots, \| x - \ell_j\|_2\big\}
    \]
    We   endow $L$ with the Euclidean distance in $\C^{2}$, and compute
    the barcodes $\bcd_{i}^\mathcal{R}(L, \|\cdot\|_2)$, $i =0 ,1,2$,  directly using   \verb"Ripser" \cite{ripser}.\\[-0.2cm]

    \item[K\"unneth]  Let $p_1, p_2 : \C^2 \longrightarrow \C$ be the projections onto the first and second coordinate, respectively.
    We select  two sets of landmarks, $X_L \subset p_1 \circ \phi(T)$ and $Y_L \subset p_2 \circ \phi(T)$ with 70 points each using \verb"maxmin" sampling,
    and endow the cartesian product $X_L \times Y_L$ with the maximum metric in $\C^2$.
    The barcodes of $X_L$ and $Y_L$ are computed using \verb"Ripser", and we deduce those of $\bcd_i^\mathcal{R}(X_L \times Y_L ,\|\cdot\|_\infty)$ for $i=0,1,2$ using Corollary \ref{cor:RipsKunneth}.

\end{description}

\subsubsection{Computational Results}
Table \ref{Table:Experiments}  shows the computational times in each approximation strategy ---
i.e., via the landmark set $L$ and the persistent K\"unneth theorem applied to $X_L \times Y_L$ ---
as well as the number of data points in each case.
\begin{table}[!htb]
\centering
\begin{tabular}{|l|c|c|}
  \hline
   & \textbf{Landmarks} & \textbf{K\"unneth} \\[.1cm]
   \hline
  Time (sec) & 15.22 & 0.20 \\[.1cm]
\hline
  \# of points & 400 & 4,900 \\[.1cm]
  \hline
\end{tabular}
\caption{Computational times, and number of points, for the landmark and K\"unneth approximations to $\bcd_i^\mathcal{R}(\mathbb{SW}_{d,\tau} f)$, $ i \leq 2$.}
\label{Table:Experiments}
\end{table}
\vspace{-.5cm}

As these results show, the K\"unneth approximation strategy is almost two orders of magnitude faster
that just taking landmarks, and
as we will see below, it is also more accurate.
Indeed,   the inequalities in (\ref{eqn:landmarkbounds}) can be used to generate a \emph{confidence region}
for the existence of an interval $J \in \bcd_i^\mathcal{R}(\mathbb{SW}_{d,\tau} f)$
given a large enough interval $I \in \bcd_i^\mathcal{R}(L )$. A similar
  region can be estimated from $\bcd_i^\mathcal{R}(X_L \times Y_L)$ as follows.
Since  the columns of $\Omega$ are orthonormal,
then for every $z\in \C^2$
we have
$
\|z \|_\infty  \leq  \|z\|_2 =  \|\Omega z\|_2  \leq   \sqrt{2} \| z \|_\infty
$
and thus
\[
R_{\epsilon}(\mathbb{SW}_{d,\tau} f , \|\cdot\|_2) \subset R_{\epsilon} (\phi(T), \|\cdot\|_\infty) \subset R_{\sqrt{2}\epsilon}(\mathbb{SW}_{d,\tau} f , \|\cdot\|_2).
\]
Moreover,  (see \ref{cor:comparisonbarcode}) if $I \in \bcd_i^\mathcal{R}(X_L \times Y_L , \|\cdot\|_\infty)$
satisfies
\[
\frac{\rho_I }{\sqrt{2}} - \sqrt{2}\ell_I \;\;>\;\; 4 d_{GH}\big(X_L\times Y_L , \phi(T)\big)
\]
where $\lambda = d_{GH}(X_L\times Y_L , \phi(T))$ is computed for subspaces of $(\C^2, \|\cdot\|_\infty)$,
then there exists a unique $J \in \bcd_i^\mathcal{R}(\mathbb{SW}_{d,\tau} f, \|\cdot\|_2)$ so that
\begin{align}\label{eqn:birthbounds}
\begin{split}
\max\left\{ 0 \; , \;
\frac{\rho_I -  2\lambda}{\sqrt{2}} \right\}\;\;\leq  & \;\rho_J \;\leq \;\;  \sqrt{2}\cdot(\rho_I + 2\lambda) \\[.1cm]
\max\left\{ 0 \; , \;
\frac{\ell_I -  2\lambda}{\sqrt{2}} \right\}\;\;\leq  & \; \ell_J \;\leq \;\;  \sqrt{2}\cdot(\ell_I + 2\lambda).
\end{split}
\end{align}

Figure \ref{fig:land_kunneth_regions}  shows the resulting confidence regions for both approximation strategies.
Each interval $[\rho , \ell)$ is replaced
by a point $(\rho , \ell) \in \R^2$, and the  regions for $\bcd_i^\mathcal{R}(L)$ (\ref{eqn:landmarkbounds})
and  $\bcd_i^\mathcal{R}(X_L \times Y_L)$ (\ref{eqn:birthbounds}) are shown in blue and red, respectively.
We compare these regions by computing their area, and report the results in Table \ref{fig:box_areas}.

\begin{figure}[!htb]
  \centering
  \includegraphics[width=0.9\textwidth]{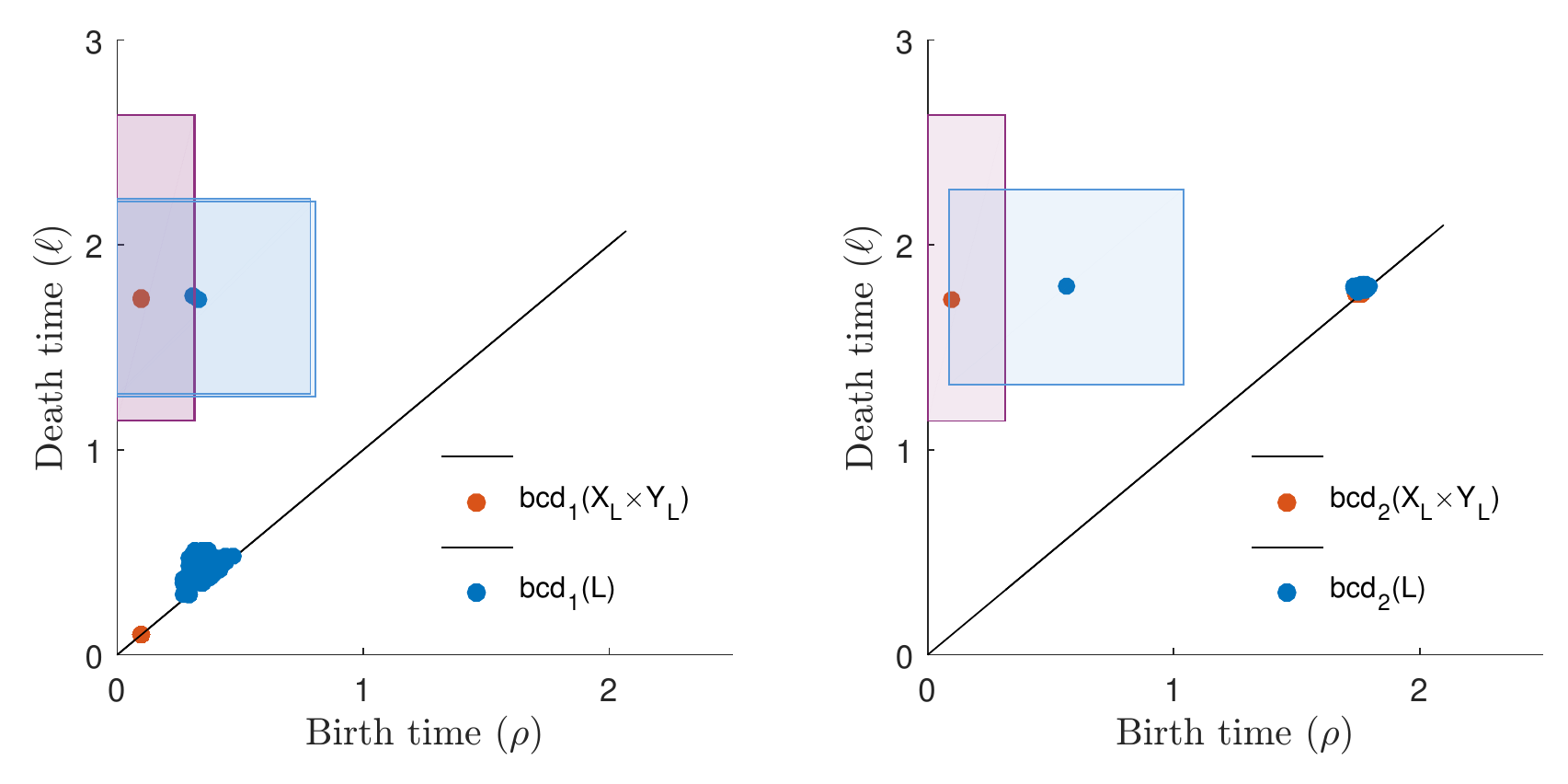}
  \caption{Confidence regions for $\bcd_i^\mathcal{R}(L)$   and $\bcd_i^\mathcal{R}(X_L \times Y_L)$, $i=1,2$. 
For  a given color, each box (confidence region) contains a point  $(\rho_J ,\ell_J)$ corresponding to a unique $J \in \bcd^\mathcal{R}_i(\mathbb{SW}_{d,\tau} f)$.}\label{fig:land_kunneth_regions}
\end{figure}

\begin{table}[!htb]
  \centering
  \begin{tabular}{|c|c|c|}
    \hline
    \textbf{$i$} & \textbf{Landmarks} & \textbf{K\"unneth} \\
\hline
    1 & 0.7677  & 0.4680 \\
\hline
    1 & 0.7480  & 0.4709 \\
\hline
    2 & 0.9072 & 0.4704 \\
    \hline
  \end{tabular}
  \caption{Areas for confidence regions from $\bcd_i^\mathcal{R}(L)$ (blue) and $\bcd_i^\mathcal{R}(X_L\times Y_L)$ (red), $i=1,2$ (see Figure \ref{fig:land_kunneth_regions}). 
A smaller area suggests a better approximation.}\label{fig:box_areas}\vspace{-.5cm}
\end{table}

These results suggest that the K\"{u}nneth strategy  provides a tighter approximation than using landmarks.
Moreover,   both  strategies can be used in tandem --- improving the quality of approximation ---  via   intersection of confidence regions.

\subsection{Application: Vietoris-Rips complexes of $n$-Tori}
The main theorem in \cite{adams} implies that the Vietoris-Rips complex ${R}_{\epsilon}(S^1_r)$ of a circle of radius $r$(equipped with Euclidean metric) is homotopy equivalent to $S^{2l+1}$ if
\begin{align*}
2r \sin\left(\pi \frac{l}{2l+1}\right) < \epsilon \leq 2r\sin\left(\pi \frac{l+1}{2l+3}\right).
\end{align*}
Consider the categorical product of Vietoris-Rips complexes on $N$ circles $ \mathcal{R}(S^1_{r_1}) \times \mathcal{R}(S^1_{r_2}) \times \cdots \times \mathcal{R}(S^1_{r_N}) $. We know it is homotopy equivalent to $\mathcal{ R}( \textbf{T} )$ where $\textbf{T} = S^1_{r_1} \times \cdots\times S^1_{r_N}$ equipped with the maximum metric. For its barcode $\mathsf{bcd}_p$ in dimension $p$, let $\mathsf{bcd}_p^{\epsilon}$ be the subcollection of bars in $\mathsf{bcd}_p$ that start at $0$ and end after $\epsilon$. For each $1 \leq n \leq N$, the only two bars in $\mathsf{bcd}_p(\mathcal{R}(S^1_{r_n}))$ that start at $0$ are the intervals $[0,\sqrt{3}r_n] \in \mathsf{bcd}_1\left(\mathcal{R}\left(S^1_{r_n}\right)\right)$ and $[0,\infty)\in\mathsf{bcd}_0\left(\mathcal{R}\left(S^1_{r_n}\right)\right)$. Then
\begin{align*}
	 {\mathsf{bcd}}_p^{\epsilon}\left( \mathcal{R}(\textbf{T})\right)= \bigg\lbrace I_1 \cap \cdots \cap  I_N \ \bigg\vert \ I_n \in  {\mathsf{bcd}}_{m_n}^{\epsilon}\left(\mathcal{R}\left(S^1_{r_n}\right)\right), m_n \in \{0,1\}, \sum_{n=1}^{N} m_n = p\bigg\rbrace
\end{align*}
Let $\chi_n$ be the indicator function for the interval $[0,\sqrt{3}r_n]$, that is, $\chi_n(\epsilon) = 1$ if $\epsilon \in [0,\sqrt{3}r_n]$ and $0$ otherwise. Then the number of elements in ${\mathsf{bcd}}_p^{\epsilon}\left( \mathcal{R}(\mathsf{T})\right)$ is given by
\begin{align*}
\#\left({\mathsf{bcd}}_p^{\epsilon}\left( \mathcal{R}(\textbf{T})\right)\right) = {\sum_{n=1}^N \chi_n(\epsilon) \choose p}.
\end{align*}
This result was originally proved in \cite[Proposition 2.7]{toroidal}. We can also compute the barcodes
\begin{align*}
\mathsf{bcd}_1(\mathcal{ R }(\textbf{T})) = \big\lbrace [0,\sqrt{3}r_n] \ \big\vert n = 1,\dots, N \big\rbrace
\end{align*}
in dimension $1$, and
\begin{align*}
\mathsf{bcd}_2(\mathcal{ R }(\textbf{T})) = \bigg\lbrace \left[0,\sqrt{3}\min\{r_n,r_m\}\right] \ \bigg\vert \ n,m = 1,\dots, N \bigg\rbrace
\end{align*}
in dimension $2$.

\bibliographystyle{acm}
\bibliography{PersistentKunnethBibliography}

\end{document}